\documentclass[a4paper,11pt]{article}
\usepackage[T1]{fontenc}
\usepackage[utf8]{inputenc}
\usepackage{amsmath,amsthm,amssymb,color,graphicx,bm,mathtools}
\usepackage{algorithm}
\usepackage{algorithmic}
\usepackage{layaureo}
\usepackage{microtype}
\usepackage{emptypage}
\usepackage{indentfirst}
\usepackage{CJKutf8}
\usepackage{enumitem}
\usepackage[T1]{fontenc}
\usepackage{lmodern}
\usepackage{multicol}

\DeclarePairedDelimiter{\abs}{\lvert}{\rvert}
\DeclarePairedDelimiter{\norm}{\lVert}{\rVert}

\usepackage[all, pdf]{xy}
\usepackage{subfig}
\usepackage{enumitem}
\usepackage{color}
\usepackage{graphicx,booktabs}
\usepackage[bookmarks,colorlinks=true,linkcolor=blue]{hyperref} 

\newcommand{\numberset}{\mathbb}
\newcommand{\N}{\numberset{N}} 
\newcommand{\R}{\numberset{R}} 
\renewcommand{\phi}{\varphi} 
\renewcommand{\chi}{\mathcal{X}} 
 
\newcommand{\Hst}{\tilde{H}^s}
\renewcommand{\epsilon}{\varepsilon}

\newcommand{\weak}{\rightharpoonup}
\newcommand{\weakstar}{\rightharpoonup^*}

\newtheorem{theorem}{Theorem}
\newtheorem{prop}[theorem]{Proposition}
\newtheorem{definition}[theorem]{Definition}
\newtheorem{lemma}[theorem]{Lemma}

\newtheorem{remark}[theorem]{Remark}

\newenvironment{system}%
{\left\lbrace\begin{aligned}}%
{\end{aligned}\right.}

\begin{document}

\title{On the obstacle problem for fractional semilinear wave equations}
\author{
	M. Bonafini\thanks{Institut f\"ur Informatik, Georg-August-Universit\"at G\"ottingen, Germany, bonafini@cs.uni-goettingen.de}
	,
	V.P.C. Le\thanks{Dipartimento di Matematica, Universit\`a di Trento, Italy, e-mail: vanphucuong.le@unitn.it}
	,
	M. Novaga\thanks{Dipartimento di Matematica, Universit\`a di Pisa, Italy, e-mail: matteo.novaga@unipi.it}
	,
	G. Orlandi\thanks{Dipartimento di Informatica, Universit\`a di Verona, Italy, e-mail: giandomenico.orlandi@univr.it}}
\date{\today}

\maketitle
\begin{abstract}
	We prove existence of weak solutions to the obstacle problem  for semilinear wave equations (including the fractional case) by using a suitable approximating scheme in the spirit of minimizing movements. This extends the results in \cite{bonafini2019variational}, where the linear case was treated. In addition, we deduce some compactness properties of concentration sets (e.g. moving interfaces) when dealing with
	singular limits of certain nonlinear wave equations.
\end{abstract}

\section{Introduction}
Semilinear wave equations have been considered extensively in the mathematical literature with many dedicated contributions (see for example \cite{Schatzman78, Maruo85, Ta94, SeTi, Neu, bellettini2010time, jerrard2011defects, del2018interface} and references therein). Our main motivation is to study certain nonlinear wave equations (possibly non-local) giving rise to interfaces (or defects) evolving by curvature such as minimal surfaces in Minkowski space: for instance, consider the class of equations
\begin{equation}\label{eq:semi1}
	\epsilon^{2}(u_{tt}-\Delta u) + \nabla_{u} W(u) = 0
\end{equation}
for $u \colon [0,\infty) \times \R^d \to \R^{m}$, where $W$ is a balanced double-well potential, $m\geq 1$, and $\epsilon>0$ is a small parameter (see for example \cite{Neu, bellettini2010time, jerrard2011defects, SmailyJerrad, del2018interface}). This case is the hyperbolic version of the stationary Allen--Cahn equation where the defects are Euclidean minimal surfaces and the parabolic Ginzburg--Landau where defects evolve according to motion by mean curvature (see for instance \cite{Modica, Tom, Giando} and references therein).

Obstacle problems in the elliptic and  parabolic setting have attracted a lot of attention including both local and non-local operators (see for example \cite{Si, CaPeSh, CaFi13, NoOk15, BaFiRo18} and references therein). In the hyperbolic scenario, we would like to mention works by Schatzman and collaborators (see for example  \cite{ Schatzman78, Schatzman80, Schatzman81, PaoliSchatzman02I}) and more recently, a work by Kikuchi dealing with the vibrating strings with an obstacle in the $1$-dimensional case by using a time semidiscrete method (see \cite{Ki09}). Notice that similar time semidiscrete methods have also been used to treat hyperbolic free boundary problems (see \cite{GiSv09, Yoshiho}). By using the same approach as in \cite{Ki09}, the obstacle problem for the fractional wave equation has been investigated in \cite{bonafini2019variational}, in which the existence of suitably defined weak solutions is proved.

In this paper, following  \cite{bonafini2019variational}, we implement a semidiscrete in time approximation scheme in order to prove existence of solutions to hyperbolic PDEs with possibly specific additional conditions. The scheme is closely related to  the concept of minimizing movements introduced by De Giorgi, and it is also elsewhere known as the discrete Morse semiflow approach or Rothe's method \cite{Ro30}. Our main focus is to prove the existence of weak solutions to the following PDEs (including also the obstacle case):
\begin{equation}\label{eq:semi2}
	\begin{system}
		& u_{tt} + (-\Delta)^s u +\nabla_{u} W(u)= 0    &\quad&\text{in } (0,T) \times \Omega                                      , \\
		& u(t,x) = 0                                	&\quad&\text{in } [0,T] \times (\R^d \setminus \Omega)                     , \\
		& u(0,x) = u_0(x)                           	&\quad&\text{in } \Omega                                                   , \\
		& u_t(0,x) = v_0(x)                         	&\quad&\text{in } \Omega                                                   , \\
	\end{system}
\end{equation}
for $\Omega \subset \R^d$ an open bounded domain with Lipschitz boundary and $W$ a continuous potential with Lipschitz continuous derivative.  For $s>0$ the operator $(-\Delta)^s$ stands for the fractional $s$-Laplacian. We prove a classical energy bound for the approximating trajectories in Proposition \ref{prop:keyestimate} and rely upon it to prove existence of a suitably defined weak solution of $\eqref{eq:semi2}$ in the obstacle-free case (Theorem $\ref{thm:main1}$) and in the obstacle case (Theorem $\ref{thm:main2}$). The approximation scheme allows us to deal with a variety of situations, including non-local fractional semilinear wave equations, and is valid in any dimension. This gives also some compactness results for concentration sets in the singular limit of $\eqref{eq:semi1}$.

The paper is organized as follows: in Section $\ref{sec:fractional}$ we briefly review some properties of the fractional Sobolev spaces and fractional Laplace operator so as to fix notations. In Section $\ref{sec:freesemiwaves}$ we introduce the approximating scheme and apply it to  fractional semilinear wave equations by means of an appropriate variational problem, prove existence result Theorem $\ref{thm:main1}$ in the obstacle-free case, and the conservative property of the solutions, namely Proposition $\ref{Conservative}$. In proposition $\ref{singularlimits}$ we prove compactness properties for the concentration sets in the singular limit of $\eqref{eq:semi1}$. In Section $\ref{sec:semiobstacle}$ we adapt the scheme to deal with the obstacle problem for fractional semilinear wave equations, and prove Theorem $\ref{thm:main2}$. Eventually, in Section $\ref{sec:numerics}$ we present an example implementing a case related to moving interfaces in a relativistic setting.
\section*{Acknowledgements}
The authors are partially supported by GNAMPA-INdAM. The first author gratefully acknowledges the support the Emmy Noether programme of the DFG, project number 403056140. We thank the anonymous referee for pointing out remarks that improved the presentation of the paper.

\section{Preliminaries}\label{sec:fractional}

Let us fix $s \geq 0$ and $m \geq 1$. Following \cite{DiNezzaPalatucciValdinoci12}, we introduce fractional Sobolev spaces and the fractional Laplacian through Fourier transform. Consider the Schwartz space $\mathcal{S}$ of rapidly decaying $C^\infty$ functions, namely $\mathcal{S}(\R^d; \,  \R^{m})$. For any $u \in \mathcal{S}(\R^d;\, \R^{m})$ denote by
\[
\mathcal{F}u(\xi) = \frac{1}{(2\pi)^{d/2}} \int_{\R^d} \textup{e}^{-\textup{i}\xi\cdot x}u(x) \,dx
\]
the Fourier transform of $u$. The fractional Laplacian operator $(-\Delta)^s \colon \mathcal{S}(\R^d; \, \R^{m}) \to L^2(\R^d; \, \R^{m})$ can then be defined, up to constants, as
\[
(-\Delta)^s u = \mathcal{F}^{-1}(|\xi|^{2s}\mathcal{F}u) \quad \text{for all }\xi \in \R^d.
\]
Given $u, v \in L^2(\R^d;\, \R^{m})$, we consider the bilinear form
\[
[u, v]_{s} = \int_{\R^{d}} (-\Delta)^{s/2}u(x) \cdot  (-\Delta)^{s/2}v(x) \, dx
\]
and the corresponding semi-norm $[u]_s = \sqrt{[u,u]_s} = ||(-\Delta)^{s/2}u||_{L^2(\R^d;   \, \R^{m})}$. Given the semi-norm $[\cdot]_s$, we define the fractional Sobolev space of order $s$ as
\[
H^s(\R^d) = \left\{ u \in L^2(\R^d ;\, \R^{m}) \,:\, \int_{\R^d} (1+|\xi|^{2s})|\mathcal{F}u(\xi)|^2\,d\xi < +\infty \right\}
\]
equipped with the norm $||u||_s = (||u||_{L^2(\R^d)}^2 + [u]_s^2 )^{1/2}$.

Fix now $\Omega \subset \R^d$ to be an open bounded set with Lipschitz boundary and define
\[
\tilde{H}^s(\Omega) = \{ u \in H^s(\R^d;\R^{m}) \,:\, u = 0 \text{ a.e. in } \R^d\setminus \Omega \},
\]
endowed with the $||\cdot||_s$ norm, and its dual $H^{-s}(\Omega) := (\Hst(\Omega))^*$. One can prove, see e.g. \cite{mclean2000strongly}, that $\tilde{H}^s(\Omega)$ corresponds to the closure of $C^\infty_c(\Omega)$ with respect to the $||\cdot||_s$ norm.

\medskip
\section{Weak solutions for the fractional semilinear wave equations}\label{sec:freesemiwaves}
We prove in this section existence of weak solutions for the fractional semilinear wave equation. The proof, as in \cite{bonafini2019variational}, is based on a constructive time-discrete variational scheme whose main ideas date back to \cite{Ro30} and which has since then been adapted to many instances of parabolic and hyperbolic equations.

Let $\Omega \subset \R^d$ be an open bounded domain with Lipschitz boundary. For $u = u(t,x):(0, T)\times \R^{d}\to \R^{m}$, let us consider the system
\begin{equation}\label{eq:freesemiwaves}
	\begin{system}
		& u_{tt} + (-\Delta)^s u +\nabla_{u} W(u)= 0    &\quad&\text{in } (0,T) \times \Omega                                   	\\
		& u(t,x) = 0                                	&\quad&\text{in } [0,T] \times (\R^d \setminus \Omega)                      \\
		& u(0,x) = u_0(x)                           	&\quad&\text{in } \Omega                                                    \\
		& u_t(0,x) = v_0(x)                         	&\quad&\text{in } \Omega                                                    \\
	\end{system}
\end{equation}
with initial data $u_0 \in \Hst(\Omega)$ and $v_0 \in L^2(\Omega):=L^2(\Omega;\, \R^{m})$ (we conventionally intend that $v_{0}=0 \mbox{ in } \R^{d} \setminus \Omega$), and a non-negative potential $W \in C^1(\R^{m};\, \R)$ having Lipschitz continuous derivative with Lipschitz constant $K>0$, i.e.,
\begin{equation}\label{conditionW}
	|\nabla W(x)-\nabla W(y)|\leq K |x-y| \quad \text{ for any $x, y \in \R^{m}$.}
\end{equation}
As we are dealing with non-local operators, the boundary condition is imposed on the whole complement of $\Omega$.
We define a weak solution of \eqref{eq:freesemiwaves} as follows:
\begin{definition}\label{def:weak}
	Let $T > 0$. We say $u = u(t,x)$ is a weak solution of \eqref{eq:freesemiwaves} in $(0,T)$ if
	\begin{enumerate}
		\item
		$
		u \in L^\infty(0,T; \Hst(\Omega)) \cap W^{1,\infty}(0,T;L^2(\Omega))$ and $u_{tt} \in L^\infty(0,T;H^{-s}(\Omega)),
		$
		\item for all $\phi \in L^{1}(0,T;\Hst(\Omega))$
		\begin{equation}\label{eq:eqweak}
			\int_{0}^T \langle u_{tt}(t), \phi(t) \rangle dt + \int_{0}^T [ u(t), \phi(t) ]_{s} \, dt +\int_{0}^T \int_\Omega \nabla_{u} W(u(t)) \cdot \phi(t) \,dxdt = 0
		\end{equation}
		with
		\begin{equation}\label{eq:u0free}
			u(0,x)=u_{0} \quad \text{ and } \quad u_{t}(0,x)=v_{0}.
		\end{equation}
	\end{enumerate}
	The energy of $u$ is defined as
	\[
	E(u(t)) = \frac12 ||u_{t}(t)||^{2}_{L^{2}(\Omega)}+\frac12 [u(t)]_{s}^2+||W(u(t))||_{L^{1}(\Omega)}, \quad t \in [0,T].
	\]
\end{definition}
\begin{remark}\label{absolutecontinuous}
In case $u_{t} \in L^{\infty}(0, T; \tilde{H}^{s}(\Omega))$, we observe that the following energy norms 
\begin{equation}
\begin{aligned}
\frac{1}{2}||u_{t}(\cdot)||^{2}_{L^{2}(\Omega)} \,:\,[0, T]&\to [0, \infty)\\
t&\mapsto \frac{1}{2} ||u_{t}(t)||^{2}_{L^{2}(\Omega)}
\end{aligned}
\end{equation}
\begin{equation}
\begin{aligned}
\frac{1}{2}[u(\cdot)]^{2}_{s} \,:\,[0, T]&\to [0, \infty)\\
t&\mapsto \frac{1}{2}[u(t)]^{2}_{s}
\end{aligned}
\end{equation}
\begin{equation}
\begin{aligned}
||W(u(\cdot))||_{L^{1}(\Omega)} \,:\,[0, T]&\to [0, \infty)\\
t&\mapsto ||W(u(t))||_{L^{1}(\Omega)}
\end{aligned}
\end{equation}
are absolutely continuous. Moreover, for a.e $t\in (0, T)$ one has:
\begin{equation}
\begin{aligned}
\frac{1}{2}\frac{d||u_{t}(t)||^{2}_{L^{2}(\Omega)}}{dt}&=<u_{tt}(t),u_{t}(t)>,\, \frac{1}{2}\frac{d[u(t)]^{2}_{s}}{dt}=[u(t),u_{t}(t)]_{s},\\
\mbox{and }\, \frac{d||W(u(t))||_{L^{1}(\Omega)}}{dt}&=\int_{\Omega}\nabla_{u}W(u(t))\cdot u_{t}(t)dt.
\end{aligned}
\end{equation}
We refer the reader to $\cite{L.C.Evans}$ for these facts.
\end{remark}
\medskip

\noindent This section is devoted to the proof of the following theorem.
\begin{theorem}\label{thm:main1}\hspace{0.5cm}
\begin{itemize}
	\item[(i)]\label{poin1theorem2} There exists a weak solution of the fractional semilinear wave equation \eqref{eq:freesemiwaves} such that it satisfies the energy inequality:
	\begin{equation}
	E(u(t))\leq E(u(0))
	\end{equation}
	for any $t\in [0, T]$.
	\item[(ii)]\label{poin2theorem2} if $u_{0}\in \tilde{H}^{2s}(\Omega)$ and $v_{0}\in \tilde{H}^{s}(\Omega)$, then there exists a solution $u$ of the equation \eqref{eq:freesemiwaves} such that $u\in  W^{1,\infty}(0, T; \tilde{H}^{s}(\Omega)), u_{t}\in  W^{1,\infty}(0, T; L^{2}(\Omega))$. Moreover, for any $t\in [0, T]$
	\begin{equation}
	E(u(t))=E(u(0)),
	\end{equation}
	i.e the energy of $u$ is conserved during the evolution.
    \item[(iii)]\label{poin3theorem2} The equation $\eqref{eq:freesemiwaves}$ has unique solution in the class:\\ 
	$X=\lbrace \, u \, | \, \mbox{u is a weak solution of } \eqref{eq:freesemiwaves}, \, u_{t}\in L^{\infty}(0, T; \tilde{H}^{s}(\Omega)) \rbrace$ in the sense that if $v, w\in X$, then 
	for each $t\in [0, T]$
	$$v(t)=w(t)\mbox{ in } \tilde{H}^{s}(\Omega).$$ In particular the solution found in point $(ii)$, since it belongs to $X$, it is unique.
\end{itemize}
\end{theorem} 
The proof relies on an extension of the approximating scheme already used in \cite{bonafini2019variational} in the linear case, where now one has to deal with the additional contribution of the (possibly non convex) potential term (the proof would simplify in case of a convex potential, as for example in \cite{Ta94}).

\subsection{Approximating scheme}\label{approximating scheme}
For $n\in \N$, set $\tau_n= T/n$ and define $t_i^n = i\tau_n$, $0\leq i \leq n$. Let $u_{-1}^n = u_0 - \tau_n v_0$, $u_0^n=u_0$ and for every $i\geq1$ let
\begin{equation}\label{eq:scheme}
	u_i^n \in \arg \min_{u \in \Hst(\Omega)} J_i^n(u) = \arg \min_{u \in \Hst(\Omega)} \left[ \int_\Omega \frac{\abs{u-2u_{i-1}^n+u_{i-2}^n}^2}{2\tau_n^2}\,dx + \frac12 [u]_s^2 +\int_{\Omega}W(u)dx \right].
\end{equation}
We can readily see, using the direct method of the calculus of variations, that each $J_i^n$ admits a minimizer in $\Hst(\Omega)$ so that $u_i^n$ is indeed well defined (notice that we are not working under uniqueness assumptions, thus we may have to choose between multiple minimizers). For any fixed $i \in \{1,\dots,n\}$, by minimality we have
\[
\frac{d}{d\varepsilon} J_i^n(u_i^n+\varepsilon \phi) |_{\varepsilon=0}=0 \quad \text{for every } \phi \in \Hst(\Omega)
\]
or, equivalently,
\begin{equation}\label{eq:ELnoobstacle}
	\int_{\Omega} (\frac{u_i^n-2u_{i-1}^n+u_{i-2}^n}{\tau_n^2})\cdot \phi \,dx + [u_i^n , \phi]_{s}+\int_{\Omega}\nabla_{u} W(u^{n}_{i}) \cdot \phi dx = 0 \quad \text{for every } \phi \in \Hst(\Omega).
\end{equation}
We define the piecewise constant and piecewise linear interpolations over $[-\tau_n,T]$ as follows:
\begin{itemize}
	\item piecewise constant interpolant
	\begin{equation}\label{eq:uhbar}
		\bar{u}^n(t,x) =
		\begin{system}
			& u_{-1}^n(x) &\quad &t=-\tau_n             \\
			& u_i^n(x) &\quad &t \in (t_{i-1}^n,t_i^n], \\
		\end{system}
	\end{equation}
	\item piecewise linear interpolant
	\begin{equation}\label{eq:uh}
		u^n(t,x) =
		\begin{system}
			& u_{-1}^n(x)                                              					&\quad &t=-\tau_n                   \\
			& \frac{t-t_{i-1}^n}{\tau_n}u_i^n(x) + \frac{t_i^n-t}{\tau_n}u_{i-1}^n(x) 	&\quad &t \in (t_{i-1}^n,t_i^n].  	\\
		\end{system}
	\end{equation}
\end{itemize}
At the same time, upon defining $v_i^n = (u_i^n-u_{i-1}^n)/\tau_n$, $0\leq i \leq n$, let $\bar{v}^{n}$ be the piecewise constant interpolation and $v^{n}$ be the piecewise linear interpolation over $[0,T]$ of the family $\{v_i^n\}_{i=0}^n$, defined similarly to \eqref{eq:uhbar}, \eqref{eq:uh}.

From \eqref{eq:ELnoobstacle}, an integration over $[0,T]$ provides
\[
\int_{0}^T \int_\Omega \left( \frac{u^n_t(t) - u^n_t(t-\tau_n)}{\tau_n} \right) \cdot \phi(t) \,dxdt + \int_{0}^T [ \bar{u}^n(t), \phi(t) ]_{s} \, dt+\int_{0}^T \int_\Omega \nabla_{u} W(\bar{u}^{n}(t))\cdot \phi(t) \,dxdt= 0
\]
for all $\phi \in L^1(0,T;\Hst(\Omega))$, which is equivalent to
\begin{equation}\label{eq:ELn}
	\int_{0}^T \int_\Omega v_t^n(t) \cdot \phi(t) \,dxdt + \int_{0}^T [ \bar{u}^n(t), \phi(t) ]_{s} \, dt+\int_{0}^T \int_\Omega \nabla_{u} W(\bar{u}^{n}(t)) \cdot \phi(t) \,dxdt = 0.
\end{equation}
The strategy in proving Theorem \ref{thm:main1} is then to consider $\eqref{eq:ELn}$, pass to the limit as $n\to \infty$ and prove that $u^{n}$ and $\bar{u}^{n}$ converge to a weak solution of $\eqref{eq:freesemiwaves}$. In order to do so, we need the following energy estimate.

\begin{prop}[Key estimate]\label{prop:keyestimate}
	The approximate solutions $\bar{u}^n$ and $u^n$ satisfy
	\[
	\frac12 \norm{ u_t^n(t) }_{ L^2(\Omega) }^{ 2 } + \frac12 [ \bar{u}^n(t) ]_{s}^{ 2 } + ||W(\bar{u}^n(t))||_{L^1(\Omega)} 
	\leq E(u(0)) + C\tau_n
	\]
	for all $t \in [0,T]$, with $C = C(E(u(0)), K, T)$ a constant independent of $n$.
\end{prop}

\begin{proof}
	For fixed $i \in \{1,\dots,n\}$, we consider equation \eqref{eq:ELnoobstacle} with test function $\phi = u_{i-1}^n-u_i^n = -\tau_n v_i^n$ to obtain
	\begin{equation}\label{eq:mainestimate1}
		\begin{aligned}
			0 &= \int_\Omega (v_{i-1}^n-v_{i}^n)\cdot v_i^n\,dx + [u_i^n, u_{i-1}^n - u_i^n]_{s}+\int_{\Omega}\nabla_{u} W(u^{n}_{i})\cdot (u^{n}_{i-1}-u^{n}_{i})dx \\
			&\leq \frac{1}{2} \int_\Omega \left[|v_{i-1}^n|^2 - |v_{i}^n|^2 \right]\, dx + \frac{1}{2} ( [u_{i-1}^n]_{s}^{ 2 } - [u_i^n]_{s}^{ 2 } )-\int_{t_{i-1}^n}^{t_i^n}\int_{\Omega}\nabla_{u} W(u^{n}_{i})\cdot v_i^n\,dxdt
		\end{aligned}
	\end{equation}
	where we used the standard inequality $2a\cdot (b-a) \leq |b|^2-|a|^2$, for $a,b \in \R^{m}$. Let us focus on the last term in the previous expression: for any $t \in (t_{i-1}^n,t_i^n]$, we write
	\[
	\begin{aligned}
	&-\int_{t_{i-1}^n}^{t_i^n}\int_{\Omega}\nabla_{u} W(u^{n}_{i})\cdot v_i^n\,dxdt = -\int_{t_{i-1}^n}^{t_i^n}\int_{\Omega}\nabla_{u} W(\bar{u}^n(t))\cdot \bar{v}^n(t) \,dxdt \\
	&= -\int_{t_{i-1}^n}^{t_i^n}\int_{\Omega}\nabla_{u} W({u}^n(t))\cdot \bar{v}^n(t) \,dxdt -\int_{t_{i-1}^n}^{t_i^n}\int_{\Omega}(\nabla_{u} W(\bar{u}^n(t))-\nabla_{u} W({u}^n(t))\cdot \bar{v}^n(t) \,dxdt
	\end{aligned}
	\]
	We recognize in the first integral a derivative, so that
	\[
	\begin{aligned}
	&-\int_{t_{i-1}^n}^{t_i^n}\int_{\Omega}\nabla_{u} W({u}^n(t))\cdot \bar{v}^n(t) \,dxdt = -\int_{t_{i-1}^n}^{t_i^n}\int_{\Omega}\frac{d}{dt}W(u^n(t)) \,dxdt = \int_\Omega \left[ W(u^n_{i-1}) - W(u_i^n)\right]\,dx
	\end{aligned}
	\]
	On the other hand, since $\bar{u}^n$ and $u^n$ are just different interpolations of the same data and $\nabla_{u} W$ is Lipschitz continuous by assumption, the second integral can be estimated as
	\[
	\begin{aligned}
	&-\int_{t_{i-1}^n}^{t_i^n}\int_{\Omega}(\nabla_{u} W(\bar{u}^n(t))-\nabla_{u} W({u}^n(t))\cdot \bar{v}^n(t) \,dxdt \leq K \int_{t_{i-1}^n}^{t_i^n}\int_{\Omega}|\bar{u}^n(t)-{u}^n(t)||\bar{v}^n(t)| \,dxdt \\
	& = K \int_{t_{i-1}^n}^{t_i^n}\int_{\Omega}|u_i^n-(u_i^n+(t-t_i^n)v_i^n)|\cdot|v_i^n| \,dxdt = \frac{\tau_n^2}{2}K \int_\Omega \left| v_i^n \right|^2 \,dx
	\end{aligned}
	\]
	Hence, inequality \eqref{eq:mainestimate1} leads to
	\[
	\begin{aligned}
	0 &\leq \frac{1}{2} \left(||v_{i-1}^n||_{L^2(\Omega)}^2 - ||v_{i}^n||_{L^2(\Omega)}^2\right) + \frac{1}{2} ( [u_{i-1}^n]_{s}^{ 2 } - [u_i^n]_{s}^{ 2 } ) \\
	&+ \int_\Omega\left[W(u^n_{i-1}) - W(u_i^n))\right]\,dx+ \frac{\tau_n^2}{2}K ||v_{i}^n||_{L^2(\Omega)}^2
	\end{aligned}
	\]
	Taking the sum for $i = 1,\dots,k$, with $1 \leq k \leq n$, we get
	\begin{equation}\label{eq:step1}
		\begin{aligned}
			E_k^n &:= \frac12 ||v_{k}^n||_{L^2(\Omega)}^2 + \frac12 [u_k^n]_{s}^{ 2 } + \int_\Omega W(u_k^n)\,dx \\
			&\leq \frac12 ||v_{0}||_{L^2(\Omega)}^2 + \frac12 [u_{0}]_{s}^{ 2 } + \int_\Omega W(u_0)\,dx + \frac{\tau_n^2}{2}K \sum_{i=1}^{k} ||v_{i}^n||_{L^2(\Omega)}^2
		\end{aligned}
	\end{equation}
	In particular, we have
	\[
	||v_{k}^n||_{L^2(\Omega)}^2 \leq 2E(u(0)) + \tau_n^2 K \sum_{i=1}^{k} ||v_{i}^n||_{L^2(\Omega)}^2
	\]
	for any $k = 1,\dots,n$. For $n$ large enough so that $(1-\tau_n^2K) > 1/2$, we write
	\begin{equation}
	||v_{k}^n||_{L^2(\Omega)}^2 \leq \frac{1}{(1-\tau_n^2 K)} \left( 2E(u(0)) + \tau_n^2 K \sum_{i=1}^{k-1} ||v_{i}^n||_{L^2(\Omega)}^2\right)
	\end{equation}
	Then, in view of the discrete Gronwall's inequality (cf. Proposition $\ref{GI}$), we obtain that
	\begin{equation}\label{eq:vbound}
		||v_{i}^n||_{L^2(\Omega)}^2 \leq \bar{C} \quad \text{for every } i = 1,\dots, n
	\end{equation}
	with $\bar{C} = \bar{C}(E(u(0)), K)$. Taking into account \eqref{eq:vbound} into \eqref{eq:step1} we finally get
	\[
	\begin{aligned}
	E_k^n = \frac12 ||v_{k}^n||_{L^2(\Omega)}^2 + \frac12 [u_k^n]_{s}^{ 2 } + \int_\Omega W(u_k^n)\,dx &\leq E(u(0)) + \frac{\tau_n^2}{2}K \sum_{i=1}^{k} \bar{C} \leq E(u(0)) + \left(\frac{T}{2}K \bar{C}\right) \tau_n
	\end{aligned}
	\]
	for every $k = 1,\dots,n$, which is the sought for conclusion.
\end{proof}
Thanks to the energy bound of Proposition \ref{prop:keyestimate} we can now provide a suitable uniform bound on $\nabla_{u} W(\bar{u}^{n})$, which is one of the main ingredients to be able to pass to the limit in \eqref{eq:ELn}.
\begin{prop}\label{unBW'}
	Let $\bar{u}^{n}$ be the piecewise constant interpolant constructed in $\eqref{eq:uhbar}$. Then, 
	$\nabla_{u} W(\bar{u}^{n}(t))$ is bounded in $L^2(\Omega)$ uniformly in $t$ and $n$.
\end{prop}
\begin{proof}
	We first observe that $\bar{u}^{n}$ is bounded in $L^{2}(\Omega)$ uniformly in $t$ and $n$. Indeed, one has
	\begin{equation}\label{eq:mainestimate2}
		\begin{aligned}
			||u^{n}(t_{2},.)-u^{n}(t_{1},.)||_{L^{2}(\Omega)}^2&=\int_{\Omega} \left|\int_{t_{1}}^{t_{2}}u^{n}_{t}(t,x)dt\right|^{2}dx \leq (t_{2}-t_{1})\int_{\Omega}\int_{t_{1}}^{t_{2}}|u^{n}_{t}(t,x)|^{2}dtdx\\
			&=(t_{2}-t_{1})\int_{t_{1}}^{t_{2}}\int_{\Omega}|u^{n}_{t}(t,x)|^{2}dxdt \leq C(t_{2}-t_{1})^{2},
		\end{aligned}
	\end{equation}
	for any $t_{1}<t_{2}$ in $[0, T]$, where we made use of Jensen's inequality, Fubini's theorem and the uniform bound on $u^{n}_{t}$ in $L^2(\Omega)$ provided by Proposition $\ref{prop:keyestimate}$. That implies that $u^{n}$ is bounded in $L^{2}(\Omega)$ uniformly in $t$ and $n$, and so is $\bar{u}^{n}$ since $\displaystyle{\lim_{n \to \infty}\sup_{t \in [0,T]} ||u^n(t,x)-\bar{u}^n(t,x)} ||^{2}_{L^{2}(\Omega)}=0$. For every fixed $t \in [0,T]$, this uniform $L^2$-bound, combined with the Lipschitz continuity of $\nabla_{u} W$ and with boundedness of $\Omega$, provides
	\begin{equation}\label{eq:L2Wp}
		\int_\Omega |\nabla_{u} W(\bar{u}^{n}(t))|^2 \, dx \leq C_1\int_\Omega (|\bar{u}^{n}(t)| + 1)^2\,dx \leq C_2
	\end{equation}
	uniformly in $t$ and $n$.
\end{proof}

We are now in the position to prove the convergence of $u^{n}$, $\bar{u}^{n}$, $W(\bar{u}^{n})$ and $\nabla_{u} W(\bar{u}^{n})$.

\begin{prop}[Convergence of $u^n$ and $v^n$]\label{prop:convun}
	There exist a subsequence of steps $\tau_n \to 0$ and a function $u \in L^\infty(0,T;\Hst(\Omega)) \cap W^{1,\infty}(0,T;L^2(\Omega))$, with  $u_{tt} \in L^\infty(0,T;H^{-s}(\Omega))$, such that
	\begin{itemize}
		\item[(i)] $u^n \to u$ in $C^0([0,T];L^2(\Omega))$,
		\item[(ii)] $u_t^n \rightharpoonup^* u_t$ in  $L^\infty(0,T;L^2(\Omega))$,
		\item[(iii)] $u^n(t) \rightharpoonup u(t)$ in $\Hst(\Omega)$ for any $t \in [0,T]$,
		\item[(iv)] $v^n \to u_t$ in $C^0([0,T];H^{-s}(\Omega))$,
		\item[(v)] $v^n_t \rightharpoonup^* u_{tt}$ in $L^\infty(0,T;H^{-s}(\Omega))$.
	\end{itemize}
\end{prop}

\begin{proof}
	The existence of a limit function $u \in L^\infty(0,T;\Hst(\Omega)) \cap W^{1,\infty}(0,T;L^2(\Omega))$ and points $(i)$, $(ii)$ and $(iii)$ follow from Proposition \ref{prop:keyestimate} combined with Ascoli-Arzelà's Theorem (for details see, e.g., \cite[Proposition 6]{bonafini2019variational}).
	
	To prove $(iv)$ and $(v)$, we observe that from \eqref{eq:ELn}, with the aid of Proposition \ref{prop:keyestimate} and Proposition \ref{unBW'}, we have that $v_t^n(t)$ is bounded in $H^{-s}(\Omega)$ uniformly in $t$ and $n$. Combining this with the $L^2$-bound on the velocities $v_i^n$, we have 
	\begin{equation}\label{eq:vnbounds}
		v^n \text{ bounded in } L^\infty(0,T;L^2(\Omega)) \text{ and in } W^{1,\infty}(0,T;H^{-s}(\Omega))
	\end{equation}
	uniformly in $t$, $n$, and at the same time, for any given $\phi \in H^{s}(\Omega)$ and for all $0 \leq t_{1} < t_{2} \leq T$, we have
	\[
	\begin{aligned}
	\int_{\Omega}(v^{n}(t_{2})-v^{n}(t_{1}))\cdot \phi dx &=
	\int_{\Omega} \int_{t_{1}}^{t_{2}}v^{n}_{t}dt \cdot \phi dx = \int_{\Omega}\int_{t_{1}}^{t_{2}} v^{n}_{t}\cdot \phi dtdx = \int_{t_{1}}^{t_{2}}\int_{\Omega} v^{n}_{t}\cdot \phi dxdt \\
	&\leq \int_{t_{1}}^{t_{2}}||v^{n}_{t}||_{H^{-s}}||\phi||_{H^{s}}dt \leq C||\phi ||_{H^{s}}(t_{2}-t_{1}).
	\end{aligned}
	\]
	Thus, there exists $v \in W^{1,\infty}(0,T;H^{-s}(\Omega))$ such that
	\[
	v^n \to v \text{ in } C^{0}([0, T]; H^{-s}(\Omega)) \quad\text{and}\quad v_t^n \rightharpoonup^* v_t \text{ in } L^{\infty}(0,T;H^{-s}(\Omega)).
	\]
	Indeed, we have $v(t) = u_t(t)$ as elements of $L^2(\Omega)$ for a.e. $t \in [0,T]$: take $t \in (t_{i-1}^n,t_i^n]$ and $\phi \in \Hst(\Omega)$, we observe that $u_t^n(t) = v^n(t^n_i)$, so that
	\[
	\begin{aligned}
	\int_{\Omega} (u_t^n(t) - v^n(t))\cdot \phi\,dx &= \int_{\Omega} (v^n(t_i^n) - v^n(t))\cdot \phi\,dx = \int_{\Omega} \left(\int_{t}^{t^n_i} v_t^n(s)\,ds\right)\cdot \phi\,dx \\
	&\leq \tau_n ||v_t^n||_{L^\infty(0,T;H^{-s}(\Omega))} ||\phi||_{s}
	\end{aligned}
	\]
	which implies, for any $\psi(t,x) = \phi(x)\eta(t)$ with $\phi \in \Hst(\Omega)$ and $\eta \in C^1_0([0,T])$, that
	\[
	\begin{aligned}
	&\int_0^T \left[ \int_\Omega (u_t(t)-v(t))\cdot \phi\,dx\right]\eta(t) \,dt = \int_0^T\int_\Omega (u_t(t)-v(t))\cdot \psi \,dxdt \\
	&= \lim_{n\to\infty} \int_0^T\int_\Omega (u^n_t(t)-v^n(t))\cdot \psi \,dxdt = \lim_{n\to\infty} \int_0^T\left[\int_\Omega (u^n_t(t)-v^n(t))\cdot \phi\,dx\right]\eta(t)\,dt \\
	&\leq \lim_{n\to\infty} \tau_n T ||v_t^n||_{L^\infty(0,T;H^{-s}(\Omega))} ||\phi||_{s} ||\eta||_{\infty} = 0.
	\end{aligned}
	\]
	This implies
	\[
	\int_\Omega (u_t(t)-v(t))\cdot \phi\,dx = 0 \quad \text{ for all }\phi\in \Hst(\Omega) \text{ and a.e. } t \in [0,T],
	\]
	which yields $v(t) = u_t(t)$ for a.e. $t \in [0,T]$. Thus, $v_t = u_{tt}$ and
	$
	u_{tt} \in L^\infty(0,T;H^{-s}(\Omega)).
	$
\end{proof}

\begin{remark}\label{rem:utpw}
	From point $(iv)$ in Proposition \ref{prop:convun} we have that $v^n \to u_t$ in $C^0([0,T];H^{-s}(\Omega))$. At the same time, due to Proposition \ref{prop:keyestimate}, $v^n(t)$ is uniformly bounded in $L^2(\Omega)$. Thus, $v^n(t) \weak u_t(t)$ in $L^2(\Omega)$, which in turn provides
	\[
	u_t^n(t) \weak u_t(t) \text{ in } L^2(\Omega) \quad \text{for any }t \in [0,T].
	\] 
\end{remark}

\begin{prop}[Convergence of $\bar{u}^n$ and $W(\bar{u}^{n})$]\label{prop:convunbar}
	Let $u$ be the limit function obtained in Proposition \ref{prop:convun}. Then, up to a subsequence,
	\begin{itemize}
		\item[(i)] $\bar{u}^n \weakstar u$ in $L^\infty(0,T;\Hst(\Omega))$,
		\item[(ii)] $\bar{u}^n(t) \weak u(t)$ in $\Hst(\Omega)$ for any $t \in [0,T]$,
		\item[(iii)] $W(\bar{u}^{n}) \to W(u)$ in $C^{0}([0, T]; L^1(\Omega))$.
	\end{itemize}
\end{prop}

\begin{proof} Regarding $(i)$ and $(ii)$ one can proceed as in \cite[Proposition 7]{bonafini2019variational}. By construction, taking into account Proposition \ref{prop:keyestimate}, we have
\begin{equation}\label{convergenceuun}
	\begin{aligned}
	&\sup_{t \in [0,T]} \int_\Omega \abs{u^n(t,x)-\bar{u}^n(t,x)}^2 \, dx = \sum_{i=1}^n \sup_{t \in [t_{i-1}^n,t_i^n]} (t-t_i^n)^2 \int_{\Omega} |v_i^n|^2\,dx \leq \tau_n^2 \sum_{i=1}^{n} ||v_i^n||_{L^2(\Omega)}^2 \leq C\tau_n
	\end{aligned}
\end{equation}
	which implies $\bar{u}^n \to u$ in $L^\infty(0,T;L^2(\Omega))$. Furthermore, again by Proposition \ref{prop:keyestimate}, $\bar{u}^n(t)$ is bounded in $\Hst(\Omega)$ uniformly in $t$ and $n$, so that we have $\bar{u}^n \weakstar u$ in $L^\infty(0,T;\Hst(\Omega))$. Thanks to point $(i)$ in Proposition \ref{prop:convun}, we also obtain pointwise convergence $\bar{u}^n(t) \weak u(t)$ in $\Hst(\Omega)$ for any $t \in [0,T]$, which is $(ii)$.
	
	For the convergence of $W(\bar{u}^{n})$, we first observe a following property of $W$:  there are positive constants $C_{1}, C_{2}$ such that
	\begin{equation}\label{growthW}
	|W(x)-W(y)|\leq (C_{1}(|x|+|y|)+C_{2})(|x-y|)
	\end{equation}
	for any $x, y \in \R^{m}$. Indeed, let $x, y\in \R^{m}$ be fixed, by the Mean Value Theorem there exists $c\in [x, y]$, here we denote $[x, y]$ the segment connecting $x$ and $y$ in $\R^{m}$, such that
	\begin{equation}
	W(x)-W(y)=\nabla W(c)\cdot (x-y).
	\end{equation}
	Thus, from the Lipshitz continuity of $\nabla W$ we deduce that
	\begin{equation}\label{eq:mainestimate2}
	\begin{aligned}
	|W(x)-W(y)|&\leq|\nabla W(c)||x-y|\\
	&\leq (C_{1}|c|+C_{2})|x-y|\\
	&\leq (C_{1}\max \lbrace |x|, |y| \rbrace+C_{2})|x-y|\\
	&\leq (C_{1}(|x|+|y|)+C_{2})|x-y|\\
	\end{aligned}
	\end{equation}
where $C_{1}, C_{2}$ are positive constants independent of $c, x, y$.\\
Then, let $t\in [0, T]$ we have
\begin{equation}\label{eq:mainestimate2}
\begin{aligned}
\int_{\Omega}|W(\bar{u}^{n} (t))-W(u(t))|dx&\leq \int_{\Omega} (C_{1}(|\bar{u}^{n}(t)|+|u(t)|+C_{2})|\bar{u}_{n}(t)-u(t)|dx\\
&\leq\int_{\Omega} (C_{1}|\bar{u}^{n}(t)+u(t)|+C_{2})|\bar{u}^{n}(t)-u(t)|dx\\
&\leq||(C_{1}|\bar{u}^{n}(t)+u(t)|+C_{2})||_{L^{2}(\Omega)}||\bar{u}^{n}(t)-u(t)||_{L^{2}(\Omega)}\\
&\leq C_{3}||\bar{u}^{n}(t)-u(t)||_{L^{2}(\Omega)},
\end{aligned}
\end{equation}
where $C_{3}$ is a constant independent of $t, n$ due to the boundedness of $\bar{u}^{n}, u^{n}$ in $L^{2}(\Omega)$ uniformly in $t, n$ and point $(i)$ in Proposition $\ref{prop:convun}$. In addition, once again from point $(i)$ in Proposition $5$ combined with $\eqref{convergenceuun}$, it implies that $\bar{u}^{n} \to u$ in $C^0([0,T];L^2(\Omega))$. So, we can conclude that $W(\bar{u}^{n}) \to W(u)$ in $C^{0}([0, T]; L^1(\Omega))$.
\end{proof}
\begin{prop}[Convergence of $\nabla_{u} W(\bar{u}^n)$]\label{prop:Wprimeconvunbar}
	Let $u$ be the limit function obtained in Proposition \ref{prop:convun}. Then, up to a subsequence, $\nabla_{u} W(\bar{u}^n) \weakstar \nabla_{u} W(u)$ in $L^\infty(0,T;H^{-s}(\Omega))$.
\end{prop}
\begin{proof}
	The same spirit of the analysis of the convergence $W(\bar{u}^{n})$ in Proposition $\ref{prop:convunbar}$, one can check that, up to a subsequence,
	\begin{equation}
		\nabla_{u} W(\bar{u}^{n})\to \nabla_{u} W(u)\mbox{ in } L^{2}((0,T)\times \Omega).
	\end{equation}
	From Proposition $\ref{unBW'}$, we observe that $\nabla_{u} W(\bar{u}^{n})$ is bounded in $H^{-s}(\Omega)$ uniformly in $t$ and $n$, this implies our conclusion.
\end{proof}

\subsection{Proof of Theorem \ref{thm:main1}}
\begin{proof}[{\bf Proof of Theorem \ref{thm:main1} $(i)$.}]

	Let $u$ be the cluster point obtained in Proposition \ref{prop:convun}, we shall prove that $u$ is a weak solution of $\eqref{eq:freesemiwaves}$. In fact, for each $n>0$, from \eqref{eq:ELn} one has
	\[
	\int_{0}^T \int_\Omega v_t^n(t)\cdot \phi(t) \,dxdt + \int_{0}^T [ \bar{u}^n(t), \phi(t) ]_{s} \, dt+\int_{0}^T \int_\Omega \nabla_{u} W(\bar{u}^{n}(t))\cdot \phi(t) \,dxdt  = 0
	\]
	for any $\phi \in L^1(0,T;\Hst(\Omega))$. Passing to the limit as $n \to \infty$, using Propositions \ref{prop:convun}, \ref{prop:convunbar}, \ref{prop:Wprimeconvunbar}, we immediately get
	\begin{equation}\label{weakequation}
		\int_{0}^T \langle u_{tt}(t), \phi(t) \rangle dt + \int_{0}^T [ u(t), \phi(t) ]_{s} \, dt+\int_{0}^T \int_\Omega \nabla_{u} W(u(t))\cdot \phi(t) \,dxdt = 0.
	\end{equation}
	The fact that $u(0) = u_0$ and $u_t(0) = v_0$ follows observing that $u^{n}(0)=u_{0}$ and $v^{n}(0)=v_{0}$ for all $n$ and that, thanks to Proposition $\ref{prop:convun}$, $u^{n}\to u$ in $C^{0}([0, T]; L^{2}(\Omega))$ and $v^{n}\to u_{t}$ in $C^{0}([0,T]; H^{-s}(\Omega))$. Finally, the verification of energy inequality is easily obtained by passing to the limit in energy estimate in Proposition $\ref{prop:keyestimate}$.
\end{proof}

In order to prove Theorem \ref{thm:main1} $(ii)$, i.e. energy conservation for the limiting solution $u$ under more regular initial data, we actually have to slightly modify the approximating scheme, as precised in the following
\begin{prop}\label{Conservative}
Let $u_{0}\in \tilde{H}^{2s}(\Omega), v_{0}\in \tilde{H}^{s}(\Omega)$, and set $u^{n}_{-1}=u_{0}-\tau_{n}v^{n}_{0}$  where $\lbrace v^{n}_{0} \rbrace_{n} \subset \tilde{H}^{s}(\Omega)$, $v^{n}_{0}\to v_{0}$ in $\tilde{H}^{s}(\Omega)$, and such that $||u_{0}-\tau_{n}v^{n}_{0}||_{\tilde{H}^{2s}(\Omega)}\leq C$ with $C$ independent of $n$. Then, let $u^{n}, \bar{u} ^{n}$ be approximate solutions of $\eqref{eq:freesemiwaves}$ satisfying the equation $\eqref{eq:ELn}$, and $u$ be a limiting solution, we have that  $u\in W^{1,\infty}(0, T; \tilde{H}^{s}(\Omega))$, $u_{t} \in W^{1,\infty}(0, T; L^{2}(\Omega))$. Moreover, for any $0\leq t_{1}<t_{2} \leq T$, the energy $E(u)$ satisfies
	\begin{equation}
		E(u(t_{1}))=E(u(t_{2})).
	\end{equation}
\end{prop}
\begin{remark}
Observe that slightly changing the approximating scheme in the initialization step, by setting $u^{n}_{-1}=u_{0}-\tau_{n}v^{n}_{0}$, doesn't affect the properties of the approximate solutions, namely the same energy estimate as in Proposition $\ref{prop:keyestimate}$ holds true, and hence Proposition $\ref{prop:convun}$, $\ref{prop:convunbar}$, $\ref{prop:Wprimeconvunbar}$ remain valid.
\end{remark}
Proposition $\ref{Conservative}$ is a consequence of the following 
\begin{lemma}\label{uniformoninitialsteps}
Let $u_{0}\in \tilde{H}^{2s}(\Omega), v_{0}\in \tilde{H}^{s}(\Omega)$ and $u^{n}_{-1}=u_{0}-\tau_{n}v^{n}_{0}$ be as in Proposition $\ref{Conservative}$. 
Then, there exists a constant $C$ independent of $n$ such that:
	\begin{equation}
		\int_{\Omega}|\frac{u^{n}_{1}-2u^{n}_{0}+u^{n}_{-1}}{\tau^{2}_{n}}|^{2}dx+[\frac{u^{n}_{1}-u^{n}_{0}}{\tau_{n}}]^{2}_{s}\leq C.
	\end{equation}
\end{lemma}
\begin{proof}
	By substituting the test function $\phi=\frac{u^{n}_{1}-2u^{n}_{0}+u^{n}_{-1}}{\tau^{2}_{n}}$ in the Euler's equation $\eqref{eq:ELnoobstacle}$ with $i=1$, we obtain that
	\begin{equation}\label{eq:improvedboundinitialstep}
		\begin{aligned}
			&\int_{\Omega} |\frac{u_1^n-2u_{0}^n+u_{-1}^n}{\tau_n^2}|^{2} \,dx + [u_1^n , \frac{u_1^n-2u_{0}^n+u_{-1}^n}{\tau_n^2}]_{s}+\int_{\Omega}\nabla_{u} W(u^{n}_{1}) \cdot \frac{u_1^n-2u_{0}^n+u_{-1}^n}{\tau_n^2} dx=0\\
			&\Longleftrightarrow \int_{\Omega}|a^{n}_{1}|^{2}dx+[\frac{u^{n}_{1}-u^{n}_{0}}{\tau_{n}}]^{2}_{s}-[\frac{u^{n}_{0}-u^{n}_{-1}}{\tau_{n}}]^{2}_{s}+[u^{n}_{-1},a^{n}_{1}]_{s}+\int_{\Omega} \nabla_{u}W(u^{n}_{1}).a^{n}_{1}dx=0\\
		\end{aligned}
	\end{equation}
	where $a^{n}_{1}=\frac{u_1^n-2u_{0}^n+u_{-1}^n}{\tau_n^2}$. It implies that
	\begin{equation}\label{improvedboundinitialstep2}
		\int_{\Omega}|a^{n}_{1}|^{2}dx+[\frac{u^{n}_{1}-u^{n}_{0}}{\tau_{n}}]^{2}_{s}\leq [\frac{u^{n}_{0}-u^{n}_{-1}}{\tau_{n}}]^{2}_{s}+||\nabla_{u}W(u_{1})||_{L^{2}(\Omega)}||a^{n}_{1}||_{L^{2}(\Omega)}+|[u^{n}_{-1},a^{n}_{1}]_{s}|
	\end{equation}
	On the other hand, we observe that $[\frac{u^{n}_{0}-u^{n}_{-1}}{\tau_{n}}]^{2}_{s}=[v^{n}_{0}]^{2}_{s}$ and 
	\begin{equation}
		\begin{aligned}
			|[u^{n}_{-1},a^{n}_{1}]_{s}|=&|\int_{\R^{d}}|\xi|^{2s}\mathcal{F}(u^{n}_{-1})(\xi) . \,\mathcal{F} (a^{n}_{1})(\xi)d\xi| \\
			\leq & \left(\int_{\R^{d}}|\xi|^{4s}|\mathcal{F}(u^{n}_{-1})(\xi)|^{2}d\xi \right)^{\frac{1}{2}} \left(\int_{\R^{d}}|\mathcal{F}(a^{n}_{1})(\xi)|^{2}d\xi\right)^{\frac{1}{2}} 
			\leq 
			[u^{n}_{-1}]_{2s}||a^{n}_{1}||_{L^{2}(\Omega)}.
		\end{aligned}
	\end{equation}
	From this observation combined with the hypothesis, Proposition $\ref{unBW'}$, and the inequality $\eqref{improvedboundinitialstep2}$, we can deduce that there exist constants $C_{1}, C_{2}$ independent of $n$ such that
	\begin{equation}\label{improvedboundinitialstep3}
		\int_{\Omega}|a^{n}_{1}|^{2}dx \leq \int_{\Omega}|a^{n}_{1}|^{2}+[\frac{u^{n}_{1}-u^{n}_{0}}{\tau_{n}}]^{2}_{s}\leq C_{1}||a^{n}_{1}||_{L^{2}(\Omega)}+C_{2}
	\end{equation}
	This gives rise to the uniform bound on $\int_{\Omega}|a^{n}_{1}|^{2}dx$, and it also follows that $[\frac{u^{n}_{1}-u^{n}_{0}}{\tau_{n}}]^{2}_{s}$ is uniformly bounded, which is the sought conclusion.
\end{proof}

We are now ready to prove Proposition $\ref{Conservative}$ and hence Theorem \ref{thm:main1}$(ii)$ :

\begin{proof}[{\bf Proof of Theorem \ref{thm:main1} $(ii)$}]

	For each $n$ fixed, let the Euler's equation at the step $i$ subtract the step $i-1$ divided by $\tau_{n}$, we obtain that
	\begin{equation}\label{eq:ELvelocity}
		\int_{\Omega} (\frac{v_i^n-2v_{i-1}^n+v_{i-2}^n}{\tau_n^2})\cdot \phi \,dx + [v_i^n , \phi]_{s}+\int_{\Omega}\frac{\nabla_{u} W(u^{n}_{i})-\nabla_{u} W(u^{n}_{i-1})}{\tau_{n}}  \cdot \phi dx = 0 \quad \text{for every } \phi \in \Hst(\Omega).
	\end{equation}
	for $i=2,\ldots,n$, and $v^{n}_{i}=\frac{u^{n}_{i}-u^{n}_{i-1}}{\tau_{n}}$, $i=0,\ldots,n$. Now, let $a^{n}_{i}=\frac{v^{n}_{i}-v^{n}_{i-1}}{\tau_{n}}$, and substituting the test function $\phi=v^{n}_{i-1}-v^{n}_{i}$ into the equation $\eqref{eq:ELvelocity}$. One has
	\begin{equation}\label{eq:ELvelocity1} 
		\int_{\Omega} (a^{n}_{i-1}-a^{n}_{i})\cdot a^{n}_{i} \,dx + [v_i^n , v^{n}_{i-1}-v^{n}_{i}]_{s}+\int_{\Omega}\frac{\nabla_{u} W(u^{n}_{i})-\nabla_{u} W(u^{n}_{i-1})}{\tau_{n}}  \cdot (v^{n}_{i-1}-v^{n}_{i}) dx = 0.
	\end{equation}
	with $i=2,\ldots,n$. From the equation $\eqref{eq:ELvelocity1}$ and due to the Lipshitz continuity condition of $\nabla W$, it follows that
	\begin{equation}\label{eq:ELvelocity2}
		\begin{aligned}
			&0\leq \int_{\Omega} (a^{n}_{i-1}-a^{n}_{i})\cdot a^{n}_{i} \,dx + [v_i^n , v^{n}_{i-1}-v^{n}_{i}]_{s}+\tau_{n}\int_{\Omega}K| v^{n}_{i}|  \cdot |a^{n}_{i}|dx\\
			&\leq \int_{\Omega} \frac{1}{2}(|a^{n}_{i-1}|^{2}-|a^{n}_{i}|^{2})\,dx + \frac{1}{2}([v_{i-1}^{n}]_{s}^{2}-[v_{i}^{n}]_{s}^{2})+\tau_{n}\int_{\Omega}K| v^{n}_{i}|  \cdot |a^{n}_{i}|dx\\
			&\leq \int_{\Omega} \frac{1}{2}(|a^{n}_{i-1}|^{2}-|a^{n}_{i}|^{2})\,dx + \frac{1}{2}([v_{i-1}^{n}]_{s}^{2}-[v_{i}^{n}]_{s}^{2})+\frac{1}{2}\tau_{n}\int_{\Omega}K(| v^{n}_{i}|^{2}  + |a^{n}_{i}|^{2})dx\\
		\end{aligned}
	\end{equation}
	Let's sum up the previous inequality for $i=2,\ldots,k$ one has
	\begin{equation}\label{eq:ELvelocity3}
		\begin{aligned}
			\int_{\Omega}|a^{n}_{k}|^{2}\,dx+[v_{k}^{n}]_{s}^{2}&\leq  \int_{\Omega}|a^{n}_{1}|^{2}\,dx+[v_{1}^{n}]_{s}^{2}+\tau_{n}K\left( \Sigma_{i=2}^{k} \int_{\Omega}(|v^{n}_{i}|^{2}  + |a^{n}_{i}|^{2})dx \right)\\
			&\leq  ||a^{n}_{1}||^{2}_{L^{2}}+[\frac{u_{1}-u_{0}}{\tau_{n}}]_{s}^{2}+\tau_{n}K \left( \Sigma_{i=2}^{k}\int_{\Omega}|a^{n}_{i}|^{2} \right)+K^{'}\tau_{n}(k-1).\\
			&\leq  C+\tau_{n}K \left( \Sigma_{i=2}^{k}\int_{\Omega}|a^{n}_{i}|^{2} \right)+K^{'}T.
		\end{aligned}
	\end{equation}
	here we have made use of the Lemma $\ref{uniformoninitialsteps}$, and the uniform bound in $L^{2}(\Omega)$ of $v^{n}_{i}$. From $\eqref{eq:ELvelocity3}$, we can deduce that
	\begin{equation}
		\int_{\Omega}|a^{n}_{k}|^{2}\,dx \leq C+\tau_{n}K \left( \Sigma_{i=2}^{k}\int_{\Omega}|a^{n}_{i}|^{2} \right)+K^{'}T
	\end{equation}
	and, then in view of Gronwall's inequality Proposition $\ref{GI}$, there exists a constant $C(T)$ such that
	\begin{equation}\label{regularityBound}
		\int_{\Omega}|a^{n}_{k}|^{2}\,dx \leq C(T)
	\end{equation}
	It also implies that $[v^{n}_{i}]^{2}_{s}$ is uniformly bounded i.e there exists a constant $C_{1}(T)$ such that
	\begin{equation}\label{regularityBound2}
		[v^{n}_{i}]^{2}_{s} \leq C_{1}(T).
	\end{equation}
	Due to uniform bounds $\eqref{regularityBound}, \eqref{regularityBound2}$, and by the analysis as in the proof of 
	Proposition $\ref{prop:convun}, \ref{prop:convunbar}$, one can show that $u\in W^{1,\infty}(0, T; \tilde{H}^{s}(\Omega))$, $u_{t}\in W^{1,\infty}(0, T; L^{2}(\Omega))$. Then, by substituting the test function $\phi=I_{[t_{1},\, t_{2}]}\times u_{t}$ in the weak equation of $u$, where $0<t_{1}<t_{2}<T$, and $I_{[t_{1},\, t_{2}]}$ is the indicator function on the time interval $[t_{1}, t_{2}]$, we obtain that
	\begin{equation}\label{uniqueness1}
	\begin{aligned}
	&\int_{t_{1}}^{t_{2}}<u_{tt}(t),u_{t}(t)>dt+\int_{t_{1}}^{t_{2}}[u(t),u_{t}(t)]_{s}dt+\int_{t_{1}}^{t_{2}}\int_{\Omega}\nabla_{u}W(u(t,x))u_{t}(t,x)dxdt=0\\
	&\Longleftrightarrow \int_{t_{1}}^{t_{2}} \frac{dE(u(t))}{dt}dt=0\\
	&\Longleftrightarrow E(u(t_{1}))=E(u(t_{2}))
	\end{aligned}
	\end{equation}
	i.e $E$ is constant inside the interval $(0, T)$, and we can extend the conservative property at endpoints by using the absolute continuity in time of $u$, $u_{t}$, and $W(u)$ in appropriate energy spaces.
\end{proof}

\begin{proof}[{\bf Proof of Theorem \ref{thm:main1}, $(iii)$}]
We are left to prove the uniqueness property: Indeed, let $v\in X$, and consider the following quantity
$$K(t)=\frac{1}{2}||u_{t}(t)-v_{t}(t)||^{2}_{L^{2}(\Omega)}+\frac{1}{2} [u(t)-v(t)]^{2}_{s}.$$
From Remark $\ref{absolutecontinuous}$, one has 
\begin{equation}\label{uniqueness1}
\begin{aligned}
&\int_{0}^{t}\frac{dK(t^{'})}{dt^{'}}dt^{'}=\int_{0}^{t}\int_{\Omega}<u_{t^{'}t^{'}}-v_{t^{'}t^{'}}, u_{t^{'}}-v_{t^{'}}>dt^{'}+\int_{0}^{t}[u(t^{'})-v(t^{'}), u_{t^{'}}(t^{'})-v_{t^{'}}(t^{'})]_{s}dt^{'}\\
&=-\int_{0}^{t}\left( \int_{\Omega}(\nabla_{u}W(u)-\nabla_{v}W(v))(u_{t^{'}}-v_{t^{'}})dx \right) dt^{'}
\end{aligned}
\end{equation}
here we have made use of the weak equation of $u$ with the test function $u_{t^{'}}\times I_{[0,\, t]}$ subtracting the one of $v$ with the test function $v_{t^{'}}\times I_{[0,\, t]}$, $I_{[0,\, t]}$ is the indicator function on the time interval $[0, t]$. From the equation $\eqref{uniqueness1}$ combined with Lipshitz continuity property of $\nabla W$ and Poincaré-type inequality in Proposition $\ref{poincareinequality}$, we obtain that
\begin{equation}\label{eq:uniquenessproperty2}
\begin{aligned}
K(t)&\leq K\int_{0}^{t}\int_{\Omega}|u-v||u_{t}-v_{t}|dx\\
&\leq \frac{1}{2}K\int_{0}^{t}\left( ||u(t)-v(t)||^{2}_{L^{2}(\Omega)}+||u_{t^{'}}(t^{'})-v_{t^{'}}(t^{'})||^{2}_{L^{2}(\Omega)}\right)dt^{'}\\
&\leq C_{s} \int_{0}^{t} \left( \frac{1}{2} [u(t^{'})-v(t^{'})]^{2}_{s}+ \frac{1}{2}||u_{t^{'}}(t^{'})-v_{t^{'}}(t^{'})||_{L^{2}(\Omega)}\right)dt^{'}\\
&\leq C_{s}\int_{0}^{t}K(t^{'})dt^{'}
\end{aligned}
\end{equation}
for some postive constants $C_{s}$, by Gronwall's inequality in Proposition $\ref{continuousGI}$, it implies that 
$$K(t)=0$$
for any $t\in [0, T]$ here we extend to the endpoint $t=T$ by using the absolute continuity in time of $u, v$. Then, it is easy to show that 
$$u(t)=v(t) \mbox{ in } \tilde{H}^{s}(\Omega)$$ for any $t\in [0, T]$.
\end{proof}

\subsection{Singular limits of nonlinear wave equations}
We turn our attention to the application of the results in the previous section to the singular limits of semilinear wave equation $\eqref{eq:semi1}$ related to topological defects (timelike minimal surfaces in Minkowski space). We consider the hyperbolic Ginzburg-Landau equation:
\begin{equation}\label{eq:hydefects}
\begin{system}
& \epsilon^{2}\left(\frac{\partial^{2} u_{\epsilon}}{\partial^{2} t}-\Delta u_{\epsilon}\right) +\nabla_{u} W(u_{\epsilon})= 0             	&\quad&\text{in } (0,T) \times \Omega, 	\\
& u_{\epsilon}(0,x) = u_{\epsilon}^{0}(x)                           &\quad&\text{in } \Omega                                                   , \\
& u_{\epsilon t}(0,x) = v^{0}_{\epsilon}(x)                         &\quad&\text{in } \Omega                                                   , \\
\end{system}
\end{equation}
where $\epsilon>0$ is a small parameter, $\Omega$ is a bounded domain in $\R^{d}$, for functions 
\begin{equation}
u_{\epsilon}: (0, T)\times \Omega\longrightarrow \R^{m},
\end{equation}
we will focus on the cases $m=1$, $m=2$, and $W$ is a non-convex balanced double well potential of class $C^{2}$. So as to apply the results in Section $\ref{sec:freesemiwaves}$, for simplicity we assume that the potential is given by
\begin{equation}
W(u)= \frac{(1-|u|^{2})^{2}}{1+|u|^{2}}.
\end{equation} 
Let us now introduce relevant quantities when dealing with topological defects: the first one is the gradient $\nabla u_{\epsilon}$ (for $m=1$), and the second is the Jacobian $2$-form, $Ju_{\epsilon}=du^{1}_{\epsilon}\wedge du^{2}_{\epsilon}$ (in the case $m=2$) defined on $(0, T)\times \Omega$. Both will be considered as distributions (concerning the distributional Jacobian, see for instance \cite{JerrardSoner, AlbertiBaldoOrlandi}) . We can prove that under natural bounds on initial energy they enjoy compactness properties and concentrate on codimension $m$ rectifiable sets in $(0, T)\times \Omega$ as $\epsilon \to 0^{+}$. We have
	\begin{prop}\label{singularlimits}
		Let $(u_{\epsilon})_{0<\epsilon<1}$ be a sequence of solutions of $\eqref{eq:hydefects}$ constructed by the approximating scheme in Section $\ref{sec:freesemiwaves}$ for each $0<\epsilon<1$ fixed such that $\frac{E(u_{\epsilon}(0))}{k_{\epsilon}} \leq C$ where $C$ is a constant independent of $\epsilon$, $k_{\epsilon}=\frac{1}{\epsilon}$ for $m=1$ and
		$k_{\epsilon}=|\log \epsilon |$ for $m=2$. Then, up to a subsequence $\epsilon_{n}\to 0$, 
		\begin{itemize}
			\item In case $m=1$, $$u_{\epsilon_{n}}\to u \mbox{ in } L^{1}((0, T)\times \Omega),$$
		where $u(t, x)\in \lbrace -1, 1 \rbrace$ for a.e. $(t,x)\in (0, T)\times \Omega$, and $u\in BV((0, T)\times \Omega)$.
		\item In case $m=2$,
		$$Ju_{\epsilon_{n}} \weak J \mbox{ in } [C^{0, 1}((0, T)\times \Omega)]^{*},$$ where $\frac{1}{\pi}J$ is a $d-1$  dimensional integral current in $(0, T)\times \Omega$.
		\end{itemize}
	\begin{proof}		
	In fact, for each $\epsilon$, from Theorem $\ref{thm:main1}$ the solution $u_{\epsilon}$ which is constructed by the approximating scheme in Section $\ref{sec:freesemiwaves}$  satisfies the energy inequality:
	\begin{equation}\label{defect}
	E(u_{\epsilon}(t))\leq E(u_{\epsilon}(0))
	\end{equation}
for any $t\in [0, T]$. Recall that $E(u_{\epsilon}(t)))=\frac{1}{2}||u_{\epsilon t}(t)||^{2}_{L^{2}(\Omega)}+\frac{1}{2}||\nabla u_{\epsilon}(t)||^{2}_{L^{2}(\Omega)}+\frac{1}{\epsilon^{2}}||W(u_{\epsilon}(t))||_{L^{1}}$.
By assumption we have
\begin{equation}\label{boundedness}
\frac{E(u_{\epsilon}(0))}{k_{\epsilon}} \leq C
\end{equation}
where $C$ is a constant independent of $\epsilon$, $k_{\epsilon}=\frac{1}{\epsilon}$ for $m=1$ and
$k_{\epsilon}=|\log \epsilon |$ for $m=2$.

Then,
\begin{itemize}
	\item In the case $m=1$, by integrating from $0$ to $T$  both side in $\eqref{defect}$ combined with $\eqref{boundedness}$ we obtain that
	\begin{equation}\label{Mordica-Mortola}
	\int_{(0, T)\times \Omega}\epsilon|  \nabla_{t,x}u_{\epsilon}(t,x) |^{2}dxdt+\int_{(0, T)\times \Omega} \frac{1}{\epsilon}W(u_{\epsilon}(t,x))dxdt \leq TC
	\end{equation}
	where $\nabla_{t,x}$ is the gradient in the space-time. In view of Modica-Mortola Theorem (see \cite{Modica}), it follows that
	there exists a function $u\in BV((0, T)\times \Omega; \lbrace-1, 1 \rbrace$) such that $u_{\epsilon}$ converges to $u$ in $L^{1}((0, T)\times \Omega)$ up to a subsequence. Moreover, the reduced boundary of the set $\Sigma^{1}=\lbrace (t,x) \in (0, T)\times \Omega \, | \, u(t,x)=1 \rbrace$ denoted by $\partial^{*}\Sigma^{1}$ is a $d-$dimensional rectifiable set in $(0, T)\times \Omega$ (for the definition of reduced boundary, see \cite{LeonSimon}). The set $\partial^{*}\Sigma^{1}$ is said to be the jump set of $u$  and it is a type of defects of the interfaces.
	\item In the complex case, following the results in $\cite{Jacobian}$, again from $\eqref{defect}$, up to a subsequence, we have that $Ju_{\epsilon} \weak J$ in $[C^{0, 1}((0, T)\times \Omega)]^{*}$, where $\frac{1}{\pi}J$ is a $d-1$  dimensional integral current in $(0, T)\times \Omega$, which concentrates on $d-1$ dimensional rectifiable set $\Sigma^{2}$ so called the vorticity set.
\end{itemize}
\end{proof}
\end{prop}
To study the dynamics of jump and the vorticity sets one has to rely on the analysis of renormalized Lagrange density
\begin{equation}\label{lagrangedensity}
\mu_{\epsilon}=\frac{\ell(u_{\epsilon}(t,x))}{k_{\epsilon}}dxdt
\end{equation}
where $\ell(u_{\epsilon})=\frac{-|u_{\epsilon t}|^{2}+|\nabla u_{\epsilon}|^{2}}{2}+\frac{W(u_{\epsilon})}{\epsilon^{2}}$. In $\cite{Neu}$, Neu showed that certain solutions of $\eqref{eq:hydefects}$ in case $m=1$ give rise to interfaces sweeping out a timelike lorentzian minimal surface of codimension $1$. Further rigorous analysis were given in ($\cite{jerrard2011defects, SmailyJerrad, del2018interface}$), where solutions of $\eqref{eq:hydefects}$ having interfaces near a given timelike minimal surface were constructed. However, due to the presence of singularities, the validity of those results are only for short times. On the other hand, the limit behavior of hyperbolic Ginzburg-Landau equation $\eqref{eq:hydefects}$ as $\epsilon \to 0^{+}$ without restricting short times (i.e also after the onset of singularities) has been treated in $\cite{bellettini2010time}$ under conditional assumptions that the measure $\mu_{\epsilon}$ is shown to concentrate on a timelike lorentzian minimal submanifold of codimension $m$ within the varifold framework developed in $\cite{BNOLorentzian}$. This has been proved by adapting the parabolic approach  $\cite{AmbSoner}$ to the hyperbolic setting through the analysis of the stress-energy tensor. We conjecture that the assumptions in $\cite{bellettini2010time}$ could be relaxed for the solutions constructed by our approximating scheme through exploiting the minimizing properties of our approximate solutions.
\newpage
\section{The obstacle problem for fractional semilinear wave equations}\label{sec:semiobstacle}
In this section, following the pipeline of \cite{bonafini2019variational}, we move on to study the obstacle problem for the fractional semilinear wave equation. From now on we assume $m = 1$ and work with real valued functions. Given an open bounded domain $\Omega \subset \R^d$ with Lipschitz boundary and a function $g \in C^0(\bar{\Omega})$, $g<0$ on $\partial \Omega$, we are interested in the obstacle problem described by
\begin{equation}\label{eq:obstaclesemiwaves}
	\begin{system}
		& u_{tt} + (-\Delta)^s u+W'(u) \geq 0       &\quad&\text{in } (0,T) \times \Omega 						\\
		& u(t,\cdot) \geq g                         &\quad&\text{in } [0,T] \times \Omega                       \\
		& (u_{tt} + (-\Delta)^s u+W'(u))(u-g) = 0   &\quad&\text{in } (0,T) \times \Omega 						\\
		& u(t,x) = 0                                &\quad&\text{in } [0,T] \times (\R^d \setminus \Omega)      \\
		& u(0,x) = u_0(x)                           &\quad&\text{in } \Omega                                    \\
		& u_t(0,x) = v_0(x)                         &\quad&\text{in } \Omega                                    \\
	\end{system}
\end{equation}
with $u_0 \in \Hst(\Omega)$, $u_0 \geq g$ a.e. in $\Omega$, and $v_0 \in L^2(\Omega)$ (with $W$ as in Section $\ref{sec:freesemiwaves}$ with $m=1$).
We define a weak solution of $\eqref{eq:obstaclesemiwaves}$ as follows:
\begin{definition}\label{def:weakobst}
	Let $T > 0$. We say $u = u(t,x)$ is a weak solution of \eqref{eq:obstaclesemiwaves} in $(0,T)$ if
	\begin{enumerate}
		\item\label{cd1} $u \in L^\infty(0,T; \Hst(\Omega)) \cap W^{1,\infty}(0,T;L^2(\Omega))$ and $u(t,x) \geq g(x)$ for a.e. $(t,x) \in (0,T)\times \Omega$;
		\item\label{cd2} there exist weak left and right derivatives $u_t^{\pm}$ on $[0,T]$ (with appropriate modifications at endpoints);
		\item\label{cd3} for all $\phi \in W^{1,\infty}(0,T;L^2(\Omega)) \cap L^1(0,T;\Hst(\Omega))$ with $\phi \geq 0$, $\text{spt}\,\phi \subset [0,T)$, we have
		\[
		-\int_{0}^{T} \int_{\Omega} u_t\phi_t \, dxdt + \int_{0}^{T} [u, \phi]_{s} \, dt+\int_{0}^{T}\int_{\Omega}W'(u)\phi dxdt - \int_\Omega v_0\,\phi(0) \, dx \geq 0
		\]
		\item\label{cd4} the initial conditions are satisfied in the following sense
		\[
		u(0,\cdot) = u_0, \quad \int_\Omega (u_t^+(0)-v_0)(\phi-u_0) \, dx \geq 0 \quad \forall \phi \in \Hst(\Omega), \phi \geq g.
		\]
	\end{enumerate}
\end{definition}

This section is then dedicated to prove the existence of such a weak solution, combining results from the previous section and extensions of arguments of \cite[Section 4]{bonafini2019variational}.

\begin{theorem}\label{thm:main2}
	There exists a weak solution $u$ of the obstacle problem for fractional semilinear wave equation \eqref{eq:obstaclesemiwaves}, and $u$ satisfies
	\begin{equation}\label{ennergyestimate2}
		\frac12||u_t^\pm(t)||_{L^2(\Omega)}^2 +\frac12 [u(t)]_{s}^2+||W(u(t))||_{L^{1}(\Omega)}\leq \frac12 ||v_{0}||^{2}_{L^{2}(\Omega)}+\frac12[u_{0}]_{s}^2+||W(u_0)||_{L^{1}(\Omega)}
	\end{equation}
	for a.e. $t\in [0, T]$.
\end{theorem}
\begin{remark}[Non-uniqueness and energy behaviour]
	\normalfont{
	The notion of weak solutions introduced in Definition \ref{def:weakobst}  can be seen as the minimal requirement we can make, i.e., to control ``upward'' variations. This leaves us with less control on the behaviour of ``downward'' moving regions, which is intended in order to allow sudden adjustments when hitting the obstacle. However, these coarse requirements lead at the same time to non-uniqueness of solutions. Generally speaking, uniqueness and in particular existence of energy preserving solutions for \eqref{eq:obstaclesemiwaves} is still an open problem, with only partial results in specific one dimensional configurations hinging on purely one dimensional arguments (see, e.g, \cite{Schatzman80} for a specific $1$d setting with local energy conservation at impacts). Within our framework a local (in space and time) energy conservation is expected whenever we are ``away'' from the obstacle (in the spirit of Proposition \ref{prop:nocontact} below), but deducing/imposing any additional condition at impact times would require the use of more technical local arguments that need further specific investigations.
	}
\end{remark}

\medskip
\subsection{Approximating scheme}
For $n\in \N$, set $\tau_n= T/n$ and define $t_i^n = i\tau_n$, $0\leq i \leq n$. Let $u_{-1}^n = u_0 - \tau_n v_0$, $u_0^n=u_0$ and define
\[
K_g := \{ u \in \Hst(\Omega) \,|\, u \geq g \text{ a.e. in } \Omega \}.
\]
For every $0 < i \leq n$, given $u^n_{i-2}$ and $u^n_{i-1}$, we define $u^{n}_{i}$ as
\[
u_i^n \in \arg \min_{u \in K_g} J_i^n(u),
\]
where $J_i^n$ is defined as in \eqref{eq:scheme}. Existence of $u^{n}_{i}$ can be obtained through the direct method of calculus of variations thanks to the convexity of $K_g$. In order to provide a variational characterization of each minimizer $u_i^n$, take $\phi \in K_g$ and consider the function $(1-\varepsilon)u_i^n + \varepsilon \phi$, which belongs to $K_g$ for any sufficiently small positive $\varepsilon$. Thus, by minimality of $u_i^n$, we have
\[
\frac{d}{d\varepsilon} J_i^n(u_i^n+\varepsilon (\phi-u_i^n)) |_{\varepsilon=0} \geq 0,
\]
which is equivalent to
\begin{equation}\label{eq:vardis}
	\int_{\Omega} \frac{u_i^n-2u_{i-1}^n+u_{i-2}^n}{\tau_n^2}(\phi-u_i^n)\,dx + [u_i^n , \phi-u_i^n]_{s}+\int_{\Omega}W'(u^{n}_{i})(\phi-u^{n}_{i})dx \geq 0 \quad \text{for all } \phi \in K_g.
\end{equation}
Moreover, because every $\phi \geq u_i^n$ is also an admissible test function, we obtain that
\begin{equation}\label{eq:vardis_simple}
	\int_{\Omega} \frac{u_i^n-2u_{i-1}^n+u_{i-2}^n}{\tau_n^2}\phi\,dx + [u_i^n , \phi]_{s}+\int_{\Omega}W'(u^{n}_{i})\phi dx \geq 0 \quad \text{for all } \phi \in \Hst(\Omega), \phi \geq 0.
\end{equation}
We define  $\bar{u}^{n}$ and $u^{n}$ to be the piecewise constant and the piecewise linear  interpolations in terms of $\{u_i^n\}_i$, just as in \eqref{eq:uhbar} and\eqref{eq:uh}; furthermore, let $v^n$ be the piecewise linear interpolant of velocities $v_i^n = (u_i^n-u_{i-1}^n)/\tau_n$, $0\leq i \leq n$. Taking into account \eqref{eq:vardis_simple}, integrating from $0$ to $T$, we obtain
\[
\int_{0}^T \int_\Omega \left( \frac{u^n_t(t) - u^n_t(t-\tau_n)}{\tau_n} \right) \phi(t) \,dxdt + \int_{0}^T [ \bar{u}^n(t), \phi(t) ]_{s} \, dt+\int_{0}^T \int_\Omega W'(\bar{u}^{n}(t))\phi (t)dxdt \geq 0
\]
for all $\phi \in L^1(0,T;\Hst(\Omega))$, $\phi(t,x) \geq 0$ for a.e. $(t,x) \in (0,T)\times \Omega$.

\medskip

\begin{remark}[Extension of the key estimate]\label{rmenergyestimate2}
By choosing the test function $\phi = u_{i-1}^n$ in \eqref{eq:vardis}, we have
	\[
	0 \leq \int_\Omega \frac{(u_{i}^n-2u_{i-1}^n+u_{i-2}^n)(u_{i-1}^n-u_{i}^n)}{\tau_n^2}\,dx + [u_i^n, u_{i-1}^n - u_i^n]_{s}+\int_{\Omega}W'(u^{n}_{i})(u^{n}_{i-1}-u^{n}_{i})dx
	\]
	and following the proof of Proposition $\ref{prop:keyestimate}$, we obtain the same energy estimate
	\[
	\frac12 \norm{ u_t^n(t) }_{ L^2(\Omega) }^{ 2 } + \frac12 [ \bar{u}^n(t) ]_{s}^{ 2 } + ||W(\bar{u}^n(t))||_{L^1(\Omega)} 
	\leq E(u(0)) + C\tau_n
	\]
	for all $t \in [0,T]$, with $C = C(E(u(0)), K, T)$ a constant independent of $n$.
\end{remark}

Given that the main energy estimate is still true, we can largely repeat the convergence proofs presented in the previous section.

\begin{prop}[Convergence of $u^n$, $\bar{u}^n$, $W(\bar{u}^{n})$, and $W'(\bar{u}^{n})$, obstacle case]\label{prop:convunobstacle}
	There exists a subsequence of steps $\tau_n \to 0$ and a function $u \in L^\infty(0,T;\Hst(\Omega)) \cap W^{1,\infty}(0,T;L^2(\Omega))$ such that
	\[
	\begin{aligned}
		&u^n \to u \text{ in } C^0([0,T];L^2(\Omega)), &\quad &u^n(t) \rightharpoonup u(t) \text{ in } \Hst(\Omega) \text{ for any } t \in [0,T], \\
		&u_t^n \rightharpoonup^* u_t \text{ in }  L^\infty(0,T;L^2(\Omega)), &\quad &\bar{u}^n \weakstar u \text{ in } L^\infty(0,T;\Hst(\Omega)).
	\end{aligned}
	\]
	Furthermore, $u(t,x)\geq g(x)$ for a.e. $(t,x) \in [0,T]\times\Omega$. Also,
	\[
	\begin{aligned}
		&W(\bar{u}^{n}) \to W(u) \text{ in } C^{0}([0, T]; L^{1}(\Omega)), \quad W'(\bar{u}^{n}) \rightharpoonup^* W'(u) \text{ in } L^\infty(0,T;H^{-s}(\Omega)).
	\end{aligned}
	\]	
		
\end{prop}

\begin{proof}
	See the proof of Propositions \ref{prop:convun}, \ref{prop:convunbar}  and \ref{prop:Wprimeconvunbar}. The fact that $u(t,x)\geq g(x)$ for a.e. $(t,x) \in [0,T]\times\Omega$ follows by the fact that $u_i^n \in K_g$ for all $n$ and $0\leq i\leq n$.
\end{proof}
Regarding the regularity of $u_{t}$, similar to what happens for the obstacle problem for the linear fractional wave equation, it is nearly impossible to expect $u_t$ to posses the same regularity as the obstacle-free case, i.e. $u_{tt} \in L^\infty(0,T;H^{-s}(\Omega))$, mainly due to dissipation of energy at the contact region with the obstacle. Nonetheless, extending the pipeline outlined in \cite[Section 4]{bonafini2019variational}, we are still able to provide some sort of higher regularity for $u_{t}$.

\begin{prop}\label{prop:FBV}
	Let $u$ be the function obtained in Proposition \ref{prop:convunobstacle} and, for any fixed $0\leq \phi\in \Hst(\Omega)$, let us define $F \colon [0,T] \to \R$ as follows
	\begin{equation}\label{eq:F}
		F(t) = \int_{\Omega}u_t(t)\phi\,dx.
	\end{equation}
	Then $F \in BV(0,T)$. Moreover, $u^n_t(t) \weak u_t(t)$ in $L^2(\Omega)$ for a.e. $t \in [0,T]$.
\end{prop}

\begin{proof}
	Consider the functions $F^n \colon [0,T] \to \R$ defined as
	\begin{equation}\label{eq:Fn}
		F^n(t) = \int_{\Omega}^{} u_t^n(t)\phi \,dx.
	\end{equation}
	where $\phi$ is fixed in $\Hst(\Omega)$ with $\phi \geq 0$. The first observation is that because $u_t^n$ is bounded in $L^2(\Omega)$ uniformly in $n$ and $t$ (see Remark $\ref{rmenergyestimate2}$),
	$||F^n||_{L^1(0,T)}$ is uniformly bounded.  Moreover, $\{ F^{n} \}_{n}$ is also uniformly bounded in $BV(0, T)$: indeed, for every fixed $n > 0$ and $0\leq i\leq n$, from \eqref{eq:vardis_simple} taking into account Remark \ref{rmenergyestimate2} and Proposition \ref{unBW'}, we can deduce that
	\begin{equation}\label{eq:ELn1}
		\begin{aligned}
			&\Bigg\lvert \int_{\Omega}^{} (v_i^n - v_{i-1}^n)\phi\,dx \Bigg\rvert - \int_\Omega (v_i^n-v_{i-1}^n)\phi\,dx \leq 2\tau_n\left|[u_i^n,\phi]_{s}+\int_{\Omega}W'(u^{n}_{i})\phi dx\right| \\
			&\leq 2\tau_n\abs{[u_i^n,\phi]_{s}}+2\tau_{n}\left|\int_{\Omega}W'(u^{n}_{i})\phi dx\right| \leq 4\tau_n C ||\phi||_s
		\end{aligned} 
	\end{equation}
	Summing over $i = 1,\dots,n$, we obtain
	\[
	\begin{aligned}
	\sum_{i=1}^{n} &\Bigg\lvert \int_{\Omega} (v_i^n - v_{i-1}^n)\phi\,dx \Bigg\rvert \leq \int_\Omega v_{n}^n\phi \, dx - \int_\Omega v_0\phi\,dx + \sum_{i=1}^{n} 4\tau_n C ||\phi||_s \\ &\leq ||v_n^n||_{L^2(\Omega)} ||\phi||_{L^2(\Omega)} + ||v_0||_{L^2(\Omega)} ||\phi||_{L^2(\Omega)} + 4 TC ||\phi||_s \leq C||\phi||_{s},
	\end{aligned}
	\]
	with $C$ independent of $n$, where we make use of the uniform bound on $||v^{n}_{i}||_{L^{2}(\Omega)}$. Thus, by Helly's selection theorem, there exists a function $\bar{F}$ of bounded variation such that $F^n(t) \to \bar{F}(t)$ for every $t \in (0,T)$ as $n \to \infty$. Taking into account that $u_t^n \weakstar u_t$ in $L^\infty(0,T;L^2(\Omega))$, one can then prove that $F(t)=\bar{F}(t)$ and thus $u^{n}_{t}(t) \rightharpoonup u_{t}(t)$ for almost every $t\in (0, T)$ (we refer to \cite[Proposition 11]{bonafini2019variational} for details).
\end{proof}

From now on we can select $u_t$ to be $$u_t(t) =  \mbox{ weak-}L^2 \mbox{ limit of } u^n_t(t),$$
which is then defined for all $t \in [0,T]$.

\begin{prop}\label{Regu2}
	Fix $0\leq \phi\in \Hst(\Omega)$ and let $F$ be defined as in \eqref{eq:F}. Then, for any $t \in (0,T)$, we have
	\[
	\lim_{r\to t^-} F(r) \leq \lim_{s\to t^+}F(s).
	\]
\end{prop}

\begin{proof}
	Because $F \in BV(0,T)$, it has right and left limits at any point. Fix  $t \in (0,T)$ and let $0<r<t<s<T$. For each $n$, let us define $r_n$ and $s_n$ such that $r \in (t^n_{r_n-1},t^n_{r_n}]$ and $s \in (t^n_{s_n-1},t^n_{s_n}]$. From \eqref{eq:Fn}, proceeding as in \eqref{eq:ELn1}, we see that
	\[
	\begin{aligned}
	F^n(s)-F^n(r) &= \int_{\Omega} (u_t^n(s)-u_t^n(r))\phi\,dx = \int_{\Omega} (v^n_{s_n}-v^n_{r_n})\phi \,dx \\
	&= \sum_{i=r_n+1}^{s_n} \int_{\Omega} (v^n_{i}-v^n_{i-1})\phi \,dx \geq -2\tau_n \sum_{i=r_n+1}^{s_n} \abs{[u_i^n,\phi]_{s}}-2\tau_n \sum_{i=r_n+1}^{s_n}\int_{\Omega} \abs{W'(u^{n}_{i})\phi}dx \\
	&\geq -4C \tau_n (s_n-r_n) ||\phi||_{s}
	\end{aligned}
	\]
	for some positive constant $C$ independent of $n$.
	Moreover, $|s-r|\geq |t^n_{s_n-1}-t^n_{r_n}|=\tau_n(s_n-1-r_n)$, thus it implies that
	\[
	F^n(s)-F^n(r) \geq -2C|s-r|\cdot||\phi||_{s}-2C\tau_n||\phi||_{s}.
	\]
	By passing to the limit $n \to \infty$ we obtain that $F(s)-F(r) \geq -2C|s-r|\cdot||\phi||_{s}$,  this yields the conclusion.
\end{proof}

We are now ready to prove the existence result, namely Theorem $\ref{thm:main2}$.
\begin{proof}[Proof of Theorem \ref{thm:main2}]
	Let $u$ be the cluster point obtained in Proposition \ref{prop:convunobstacle}. It is easy to see that the first condition in  Definition \ref{def:weakobst} follows from  Proposition \ref{prop:convunobstacle}. From Proposition \ref{prop:FBV}, it implies that for any fix $\phi \geq 0,\,  \phi \in \tilde{H}^{s}$,  $F(t)=\int_{\Omega}u_{t}(t)\phi dx$ has the left and right limits for any $t\in [0, T]$ since $F$ is $BV(0, T)$, this in turn implies the second condition in Definition \ref{def:weakobst}.
	Let us verify the third and the fourth conditions in  Definition \ref{def:weakobst}.
	
	\emph{(3.)} For $n > 0$ and any test function $\phi \in W^{1,\infty}(0,T;L^2(\Omega)) \cap L^1(0,T;\Hst(\Omega))$, with $\phi \geq 0$, $\text{spt}\,\phi \subset [0,T)$, we recall that
	\begin{equation}\label{IME}
		\int_{0}^T \int_\Omega \left( \frac{u^n_t(t) - u^n_t(t-\tau_n)}{\tau_n} \right) \phi(t) \,dxdt + \int_{0}^T [ \bar{u}^n(t), \phi(t) ]_{s} \, dt+\int_{0}^{T}\int_{\Omega}W'(\bar{u}^{n})(t)\phi(t)dxdt \geq 0
	\end{equation}
	From  Proposition \ref{prop:convunobstacle}, we have
	\[
	\begin{aligned}
	\int_{0}^{T} [\bar{u}^n(t), \phi(t)]_{s} \, dt \to \int_{0}^{T} [u(t), \phi(t)]_{s} \, dt &\quad \text{as } n \to \infty,\\
	\int_{0}^{T}\int_{\Omega}W'(\bar{u}^{n}(t))\phi(t)dxdt   \to \int_{0}^{T}\int_{\Omega}W'(u(t))\phi(t)dxdt   &\quad \text{as } n \to \infty.
	\end{aligned}
	\]
	To deal with the first term of $\eqref{IME}$, we observe that
	\[
	\begin{aligned}
	\int_{0}^{T} \int_\Omega \frac{u^n_t(t)-u^n_t(t-\tau_n)}{\tau_n}\,\phi(t) \,dxdt 
	&= \int_{0}^{T-\tau_n}\int_\Omega u^n_t(t) \left( \frac{\phi(t)-\phi(t+\tau_n)}{\tau_n} \right) \, dxdt \\
	&- \int_{0}^{\tau_n} \int_\Omega \frac{v_0}{\tau_n}\,\phi(t) \, dxdt + \int_{T-\tau_n}^{T} \int_\Omega \frac{u^n_t(t)}{\tau_n}\,\phi(t) \, dxdt \\
	& \to \int_{0}^{T} \int_\Omega u_t(t)(-\phi_t(t)) \, dxdt - \int_\Omega v_0\,\phi(0) \, dx + 0 \quad \text{as } n \to \infty
	\end{aligned}
	\]
	and this completes the proof of condition ($\ref{cd3}$).
	
	\emph{(4.)} By the convergence of $u^n$ to $u$ in $C^0([0,T];L^2(\Omega))$ and $u^{n}(0)=u_{0}$, it implies that $u(0)=u_{0}$. So as to check the initial condition on velocity we assume, without loss of generality, that the sequence $u^n$ is constructed by taking $n \in \{2^m \,: m > 0\}$ (each successive time grid is obtained dividing the previous one). Fix $n$ and $\phi \in K_g$, let $T^* = m\tau_n$ for $0 \leq m \leq n$ (i.e. $T^*$ is a ``grid point'').
	We have
	\[
	\begin{aligned}
	& \int_0^{T^*} \int_{\Omega}^{} \frac{u_t^n(t)-u_t^n(t-\tau_n)}{\tau_n} (\phi - \bar{u}^n(t)) = \sum_{i=1}^m \int_{t_{i-1}^n}^{t_i^n} \int_\Omega \frac{u_i^n-2u_{i-1}^n+u_{i-2}^n}{\tau_n^2}(\phi-u_i^n) \\
	& = \int_{\Omega}^{} \sum_{i=1}^m \frac{u_i^n-2u_{i-1}^n+u_{i-2}^n}{\tau_n}(\phi-u_i^n) = \int_{\Omega}^{} \sum_{i=1}^m (v_i^n-v_{i-1}^n)(\phi-u_i^n) \\
	&= -\int_{\Omega}^{} v_0^n(\phi-u_1^n)\,dx + \int_\Omega v_m^n(\phi-u_m^n) \,dx + \tau_n \sum_{i=1}^{m-1} \int_\Omega v_i^nv_{i-1}^n\,dx \\
	&= -\int_{\Omega}^{} v_0(\phi-u_n(\tau_n))\,dx + \int_\Omega u_t^n(T^*)(\phi-u^n(T^*)) \,dx + \tau_n \sum_{i=1}^{m-1} \int_\Omega v_i^nv_{i-1}^n\,dx. \\
	\end{aligned}
	\]
	which combined with \eqref{eq:vardis} gives
	\[
	\begin{aligned}
	&-\int_{\Omega}^{} v_0(\phi-u_n(\tau_n))\,dx + \int_\Omega u_t^n(T^*)(\phi-u^n(T^*)) \,dx \geq
	-\tau_n \sum_{i=1}^{m-1} \int_\Omega v_i^nv_{i-1}^n\,dx \\
	&- \tau_n \sum_{i=1}^{m} [u_i^n, \phi-u_i^n]_{s}
	- \tau_n \sum_{i=1}^{m} \int_{\Omega}W'(u_i^n)( \phi-u_i^n)
	\geq -CT^* - CT^* ||\phi||_{s}
	\end{aligned}
	\]
	thanks to the boundedness of $W'(u_i^n)$ in $L^2(\Omega)$. Passing to the limit as $n \to \infty$, using $u^n(\tau_n) \to u(0)$ and  $u_t^n(T^*) \rightharpoonup u_t(T^*)$ (as noticed before we choose $u_{t}$ being the weak-$L^{2}$ limit of $u^{n}_{t}$), we obtain that
	\[
	\begin{aligned}
	&-\int_{\Omega}^{} v_0(\phi-u(0))\,dx + \int_\Omega u_t(T^*)(\phi-u(T^*)) \,dx \geq -CT^* - C||\phi||_{s} T^*.
	\end{aligned}
	\]
	Let $T^*$ tend to $0$ along a sequence of ``grid points'', we have that
	\[
	\int_\Omega (u_t^+(0)-v_0)(\phi-u(0)) \,dx \geq 0.
	\]
	To complete the proof, we observe that the energy estimate $\eqref{ennergyestimate2}$ is obtained by passing to the limit as $n \to \infty$ in
	\[
	\frac12 \norm{ u_t^n(t) }_{ L^2(\Omega) }^{ 2 } + \frac12 [ \bar{u}^n(t) ]_{s}^{ 2 } + ||W(\bar{u}^n(t))||_{L^1(\Omega)} 
	\leq E(u(0)) + C\tau_n
	\]
	for all $t \in [0,T]$, with $C$ a constant independent of $n$ (cf. Remark $\ref{rmenergyestimate2}$).
\end{proof}

We end this section by an observation that in the case $s=1$ the solutions become more regular whenever the approximation $u^{n}$ lies strictly above $g$.
\begin{prop}[Regions without contact]\label{prop:nocontact}
	Let $s = 1$ and, for $\delta > 0$, suppose there exists an open set $A_\delta \subset \Omega$ such that $u^n(t,x) > g(x) + \delta$ for a.e. $(t,x) \in (0,T)\times \Omega$ and for all $n > 0$. Then $u_{tt} \in L^\infty(0,T;H^{-1}(A_\delta))$ and $u$ satisfies \eqref{eq:eqweak} for all $\phi \in L^{1}(0,T;H^1_0(A_\delta))$.
\end{prop}

\begin{proof}
	Fix $n>0$ and $1 \leq i \leq n$. For every $\phi \in H^1_0(\Omega)$ with $\textup{spt}\,\phi \subset A_\delta$ we have $u^n_i + \varepsilon\phi \in K_g$ for $\varepsilon$ sufficiently small. In particular, inequality \eqref{eq:vardis_simple} turns into
	\begin{align}\label{eq:freelocalized}
	\int_{\Omega} \frac{u_i^n-2u_{i-1}^n+u_{i-2}^n}{\tau_n^2}\phi\,dx + \int_\Omega \nabla u_i^n \cdot  \nabla \phi\,dx + \int_\Omega W'(u_i^n)\phi \,dx= 0
	\end{align}
	The equality allows us to rescue the second part of the proof of Proposition \ref{prop:convun}: in the same notation, we can prove $v^n_t(t)$ to be bounded in $H^{-1}(A_\delta)$ uniformly in $t$ and $n$ by using the uniform bound on $||W'(u^{n}_{i})||_{L^{2}(\Omega)}$ provided by Proposition \ref{unBW'}. Thus, $v \in W^{1,\infty}(0,T;H^{-1}(A_\delta))$ and
	\[
	v^n \rightharpoonup^* v \text{ in } L^\infty(0,T;L^2(A_\delta)) \quad\text{and}\quad v^n \rightharpoonup^* v \text{ in } W^{1,\infty}(0,T;H^{-1}(A_\delta)).
	\]
	A localization on $A_\delta$ proves that $v_t = u_{tt}$ in $A_\delta$ so that
	\[
	u_{tt} \in L^\infty(0,T;H^{-1}({A_\delta})).
	\]
	To get \eqref{eq:eqweak} we pass to the limit in \eqref{eq:freelocalized} as we have done in the proof of Theorem $\ref{thm:main1}$.
\end{proof}

\section{A numerical example}
\label{sec:numerics}

We present in this section a simple example implementing the scheme of Section \ref{sec:freesemiwaves} for a two dimensional radially symmetric problem related to moving interfaces in the relativistic setting.
We consider equation \eqref{eq:freesemiwaves} with potential
\[
W(u) = \frac{(1-u^2)^2}{1+u^2}
\]
and a radially symmetric initial datum $u_0$ having a sharp transition at a given radius $R_0 > 0$ (the function transitioning form $1$ inside to $-1$ outside). The initial velocity is assumed to be zero and the computational domain $\Omega = B(0,\bar{R})$ for $\bar{R}>R_0$. From results in \cite{jerrard2011defects}, the solution $u(t,\cdot)$ is expected to keep the initial structure of a radially symmetric function with a sharp transition region, with said transition region evolving inwards: for $0 \leq t < R_0\pi/2$ the solution $u(t,\cdot)$ will display its transition region along the circle of radius
\[
R(t) = R_0\cos\left(\frac{t}{R_0}\right).
\]
Thus, we incorporate the radial symmetry in the minimization of \eqref{eq:scheme} and we translate the problem into a $1$d optimization over $\bar{\Omega} = [0, \bar{R}]$ and assume Dirichlet boundary conditions $\pm1$ at $0, \bar{R}$. We employ the same discretization used in \cite{bonafini2019variational}, based on classical piecewise linear finite elements. The finite dimensional optimization problem is then solved via a projected gradient descent method combined with a dynamic adaptation of the descent step size. We display the results in Figure \ref{fig:1d_exe}: we can see how the solution evolves the transition region in time and how the position of the transition follows closely the expected radius.

\begin{figure}[tbh]
	\centering
	\begin{tabular}{cc}
		\includegraphics[width=0.4\linewidth]{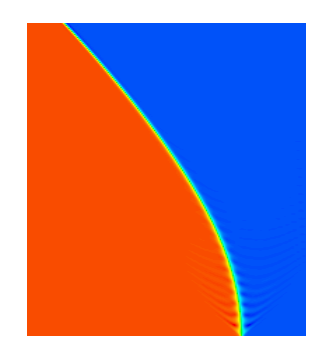}
		\includegraphics[width=0.6\linewidth]{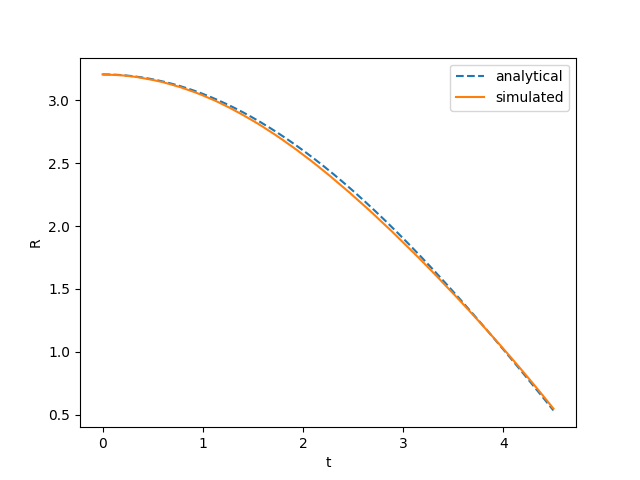}
	\end{tabular}
	\caption{Left: space-time orthogonal view of the solution, red being $+1$ and blue being $-1$. Right: time evolution of the transition region (analytical vs. simulated).}
	\label{fig:1d_exe}
\end{figure}

\newpage
\section{Appendix}
We recall the proof for the discrete Gronwall's inequality as used in the proof of Proposition $\ref{prop:keyestimate}$.
\begin{prop}(Discrete Gronwall inequality)\label{GI}
	Let $\{y_{n}\}_{n=0}^N$ be a sequence of non-negative numbers, and assume there exist two positive constants $A, B > 0$ such that
	$$
	y_{0}=0 \quad \text{and} \quad y_{n}\leq A+\frac{B}{N}\sum_{j=0}^{n-1}y_{j}\quad \text{ for all } n = 1, \dots, N.
	$$
	Then,
	\[
	y_{i}\leq A\exp(B) \quad \text{ for all } i = 1,\dots, N.
	\]
\end{prop}
\begin{proof}
	We first prove by induction that
	\begin{equation}\label{GR1}
		y_{i}\leq A\left(1+\frac{B}{N}\right)^{i}
	\end{equation}
	for all $0\leq i \leq N$. The case $i=0$ is obvious. Now suppose that $\eqref{GR1}$ holds from $0$ to $k$, then
	\[
	\begin{aligned}
	y_{k+1} &\leq A+\frac{B}{N}\sum_{j=0}^{k}y_{j}\\
	&\leq A+\frac{B}{N}\left(A\left(1+\frac{B}{N}\right)+A\left(1+\frac{B}{N}\right)^{2}+\ldots+A\left(1+\frac{B}{N}\right)^{k}\right)\\
	&=A+\frac{AB}{N}\left(\frac{\left(1+\frac{B}{N}\right)^{k}-1}{\frac{B}{N}}\right)\left(1+\frac{B}{N}\right)=A+A\left(\left(1+\frac{B}{N}\right)^{k}-1\right)\left(1+\frac{B}{N}\right)\\
	&=A\left(\left(1+\frac{B}{N}\right)^{k+1}-\frac{B}{N}\right)\leq A\left(1+\frac{B}{N}\right)^{k+1}.
	\end{aligned}
	\]
	This yields $\eqref{GR1}$, which in turn gives
	\[
	\begin{aligned}
	y_{i} &\leq A\left(1+\frac{B}{N}\right)^{i} \leq A\left[\exp\left(\frac{B}{N}\right)\right]^i = A\exp\left(\frac{i}{N}B\right)\leq A\exp\left(B\right)
	\end{aligned}
	\]
	for all $i = 0,\dots, N$.
\end{proof}
We also provide here the continuous version of Gronwall's inequality:
\begin{prop}\label{continuousGI}
Let $g: [0, 1]\longrightarrow \R$ be a non-negative continuous function, and it satisfies the following inequality:
$$g(t)\leq C\int_{0}^{t}g(s)ds$$
for any $t \in [0, 1]$, for some positive constants $C$. Then, $g(t)=0$ for any $t\in [0, 1]$.
\begin{proof}
Let $m(t)=e^{-Ct}\int_{0}^{t}g(s)ds$. We observe that 
$\frac{dm(t)}{dt} \leq 0$ for any $t\in (0, 1)$, and $m(0)=0$. Therefore, we have
$$m(t)=0$$
for any $t\in (0, 1)$, this implies that $g(t)=0$ for each $t\in (0, 1)$, and we extend to the endpoints by continuity of $g$.
\end{proof}
\end{prop}
The last inequality presented here which is the Poincaré-type inequality, is used in the proof of uniqueness.
\begin{prop}\label{poincareinequality}
Let $u\in \tilde{H}^{s}(\Omega)$. Then, there exists a postive constant $C_{s}$ such that 
$$||u||_{L^{2}(\Omega)}\leq C_{s} [u]_{s}.$$
\begin{proof}
Let $u\in \tilde{H}^{s}(\Omega)$, by Heisenberg-Pauli-Weyl inequality, one has
\begin{equation}\label{eq:uniquenessproperty}
\begin{aligned}
||u||^{4}_{L^{2}(\Omega)}&\leq C_{s}\left( \int_{\R^{d}}|x|^{2s}|u(x)|^{2}dx \right)\left( \int_{\R^{d}}|\eta|^{2s}|\mathcal{F}(u)(\eta)|^{2}d\eta \right)\\
&\leq D_{s}||u||^{2}_{L^{2}(\Omega)}[u]^{2}_{s}.
\end{aligned}
\end{equation}
for some constants $D_{s}$, we have used that $u$ is equal to $0$ outside the bounded domain $\Omega$. Thus, we obtain that
$$||u||_{L^{2}(\Omega)}\leq D_{s} [u]_{s},$$
which is the conclusion.
\end{proof}
\end{prop}

\end{document}